\documentclass{amsart}
\usepackage{amsmath}
\usepackage{amsthm}
\usepackage{amsbsy}
\usepackage{amsopn}
\usepackage{amssymb}
\usepackage{amsfonts}
\usepackage[official ]{eurosym}

\newcommand{\cH}{\mathcal{H}}

\newcommand{\cN}{\mathcal{N}}
\newcommand{\cP}{\mathcal{P}}

\newcommand{\bN}{\mathbb{N}}

\newcommand{\bR}{\mathbb{R}}

\usepackage[dvipsnames,svgnames,x11names,hyperref]{xcolor}
\usepackage{url,graphicx,verbatim,amssymb,enumerate,stmaryrd}
\usepackage[pagebackref,colorlinks,citecolor=Bittersweet,linkcolor=Bittersweet,urlcolor=Bittersweet,filecolor=Bittersweet]{hyperref}
\usepackage{tikz} 
\usetikzlibrary{arrows,decorations.pathmorphing,backgrounds,fit,positioning,shapes.symbols,chains,calc}
\usepackage[all]{xy}
\usepackage[hmargin=3cm,vmargin=3cm]{geometry}
\usepackage{setspace,kantlipsum}

\usepackage{booktabs}

\newtheorem{theorem}{Theorem}[section]
\newtheorem{lemma}[theorem]{Lemma}

\theoremstyle{definition}
\newtheorem{definition}[theorem]{Definition}

\newtheorem{example}[theorem]{Example}
\newtheorem{inductive lemma}[theorem]{Inductive Lemma}
\newtheorem{remark}[theorem]{Remark}

\title[]{\Large Refinements of the holonomic approximation lemma}

\begin{document}
\maketitle

\begin{center}
\Large  Daniel \'Alvarez-Gavela \let\thefootnote\relax\footnote{The author was partially supported by NSF grant DMS-1505910.} \\  $ $  \\
\large \itshape Stanford University
\end{center}
\begin{abstract}
The holonomic approximation lemma of Eliashberg and Mishachev is a powerful tool in the philosophy of the $h$-principle. By carefully keeping track of the quantitative geometry behind the holonomic approximation process, we establish several refinements of this lemma. Gromov's idea from convex integration of working `one pure partial derivative at a time' is central to the discussion. We give applications of our results to flexible symplectic and contact topology.
 \end{abstract}
\onehalfspacing
\tableofcontents

\section{Introduction and main results}
\subsection{Classical holonomic approximation}\label{Classical holonomic approximation}
We begin by briefly recalling the holonomic approximation lemma, which is the starting point of our paper. Given a fibre bundle $p:X \to M$, where $X$ and $M$ are smooth manifolds, consider the bundle $p^r: X^{(r)} \to M$ of $r$-jets of $p$. The fibre of $p^r$ over a point $x \in M$ consists of equivalence classes of germs of sections $h:Op(x)  \to X$ of $p$, where two germs are identified if they agree up to order $r$ at the point $x$. Throughout we use Gromov's notation $Op(A)$ to denote an arbitrarily small but unspecified open neighborhood of a subset $A \subset M$. If $X=M \times N$ is a trivial bundle, then sections of $p$ are the same as maps from $M$ to $N$ and we usually write $X^{(r)}=J^r(M,N)$. 

Given a section $h:U \to X$ of $p$ defined over an open subset $U \subset M$, we denote by $j^{r}(h):U \to X^{(r)}$ the section of $p^r$ consisting of the $r$-jets of $h$. Sections of $p^r$ of the form $j^{r}(h)$ are called holonomic.  It is generally impossible to globally approximate an arbitrary section $s:M \to X^{(r)}$ of $p^r$ by a holonomic section. Nevertheless, the approximation can always be achieved in a deformed neighborhood of any reasonable stratified subset of positive codimension. For simplicity we restrict our discussion to the following class of stratified subsets.
 
 \begin{definition} A closed subset $K \subset M$ is called a polyhedron if it is a subcomplex of some smooth triangulation of $M$. 
 \end{definition}
 
 The classical holonomic approximation lemma was first stated and proved in \cite{EM01} by Eliashberg and Mishachev. They give numerous applications in their book \cite{EM02}. Holonomic approximation is closely related to the method of continuous sheaves discovered by Gromov in his thesis \cite{G69} and further explored in his book \cite{G86}. Both of these techniques greatly generalize the immersion theory of Smale-Hirsch-Phillips \cite{H59}, \cite{P67}, \cite{S59}. The precise statement that we wish to recall reads as follows.
 
 \begin{theorem}[holonomic approximation lemma]\label{classical holonomic approximation}
 Let $s:M \to X^{(r)}$ be a section of the $r$-jet bundle of a fibre bundle $p:X \to M$ and let $K \subset M$ be a polyhedron of positive codimension. Then there exists an isotopy $F_t:M \to M$ and a holonomic section $\hat{s}:Op\big(F_1(K) \big) \to X^{(r)}$ such that the following properties hold.
 \begin{itemize}
 \item $\hat{s}$ is $C^0$-close to $s$ on $Op\big(F_1(K)\big)$.
 \item $F_t$ is $C^0$-small.
 \end{itemize}
 \end{theorem}
 \begin{remark} $ $
 \begin{enumerate} 
 \item More precisely, the $C^0$-closeness statement means that for any choice of Riemannian metric on $X$ and for any $\varepsilon>0$ there exist $F_t$ and $\hat{s}$ as in the statement of the theorem such that $\text{dist}_{C^0}(\hat{s}, s)< \varepsilon$ on $Op\big(F_1(K) \big)$ with respect to the choice of metric. Similar remarks apply below whenever we talk about the $C^0$-closeness of two maps.
\item We say that an isotopy $F_t$ is $C^0$-small if it is $C^0$-close to the identity.
 
 \item The holonomic approximation lemma also holds in relative and parametric form, see \cite{EM02} for details. 
 \end{enumerate}
 \end{remark}

\begin{figure}[h]
\includegraphics[scale=0.6]{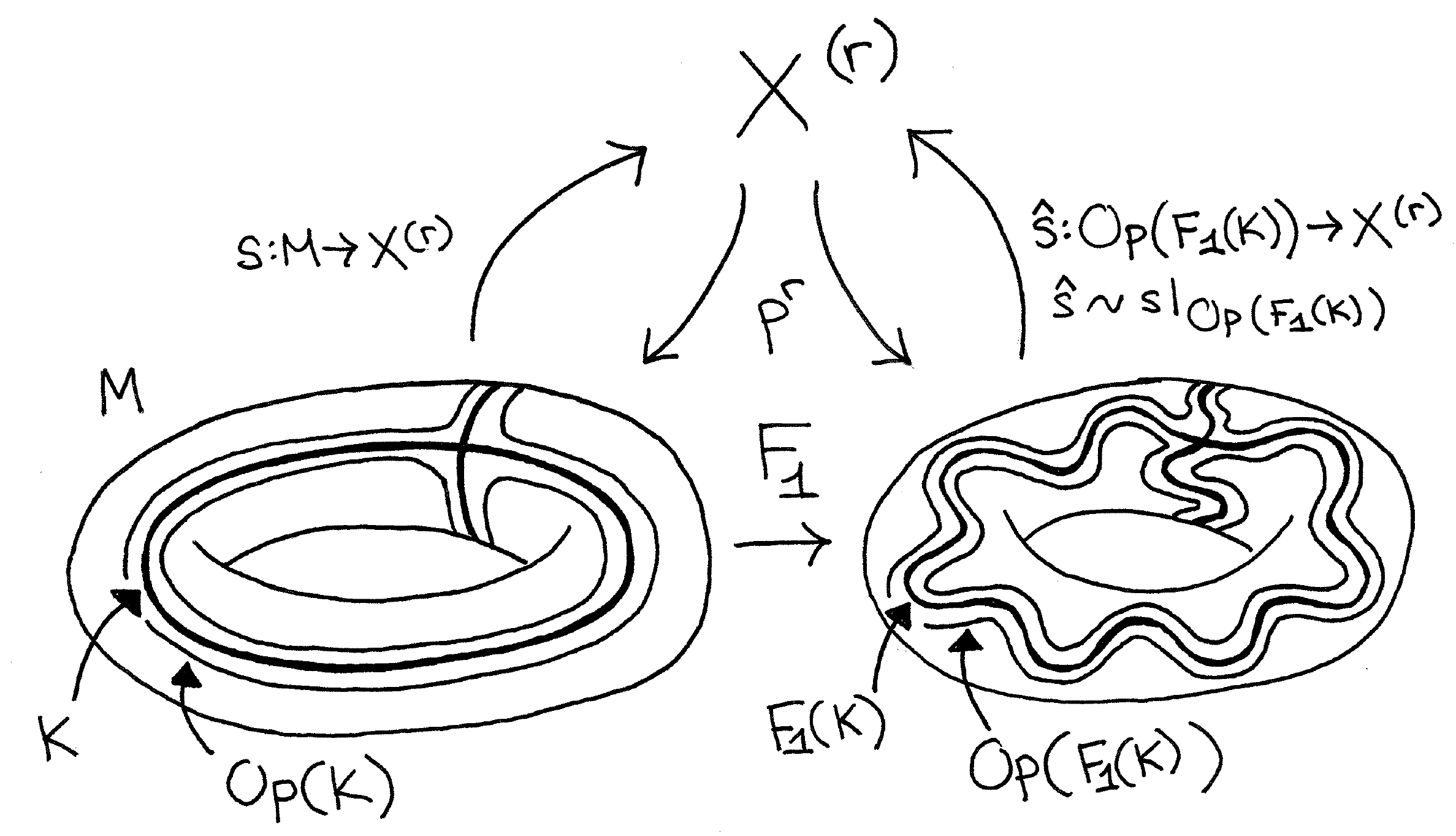}
\caption{The isotopy $F_t$ is $C^0$-small but typically wiggles the subset $K$ wildly inside $M$.}
\label{holonomic approximation}
\end{figure}

\subsection{Improved holonomic approximation}\label{Improved holonomic approximation} The holonomic approximation lemma \ref{classical holonomic approximation} is an extremely useful tool for proving $h$-principles. However, it turns out that for certain applications a stronger statement is needed. The goal of this paper is to prove several refinements of the method of holonomic approximation, which we formulate in this section. In Section \ref{Parametric versions} we state the parametric versions. In Section \ref{Applications to symplectic topology} we explain how these refinements yield new flexibility results in symplectic and contact topology. 

Recall that for $0 \leq l <r$ there are maps $p^r_l:X^{(r)} \to X^{(l)}$ which forget higher order information. The projection $p^r_l$ gives $X^{(r)}$ the structure of an affine bundle over $X^{(l)}$. Given a section $s:A \to X^{(r)}$ of $p^r$ defined over any subset $A \subset M$, we call $s^{(l)}=p^r_l \circ s : A \to X^{(l)}$ the $l$-jet component of $s$. 
\begin{definition} A section $s: U \to X^{(r)}$ of $ \,p^r$ defined over an open subset $U \subset M$ is called $l$-holonomic if $s^{(l)}$ is a holonomic section of $p^l$. 
\end{definition}
Our first refinement of the holonomic approximation lemma states that if we start with an $l$-holonomic section $s$ of $p^r$, then it is possible to carry out the holonomic approximation process on $s$ while ensuring global control of the $l$-jet component.

 \begin{theorem}[holonomic approximation lemma for $l$-holonomic sections]\label{Approximation of $l$-holonomic sections}
 Let $s:M \to X^{(r)}$ be a section of the $r$-jet bundle of a fibre bundle $p:X \to M$ and let $K \subset M$ be a polyhedron of positive codimension. Suppose that for some $l< r$ the section $s$ is $l$-holonomic. Then there exists an isotopy $F_t:M \to M$ and a holonomic section $\hat{s}:M \to X^{(r)}$ such that the following properties hold.
 \begin{itemize}
 \item $\hat{s}$ is $C^0$-close to $s$ on $Op\big(F_1(K)\big)$.
 \item $\hat{s}^{(l)}$ is $C^0$-close to $s^{(l)}$ on all of $M$.
 \item $F_t$ is $C^0$-small.
 \item $F_t=id_M$ and $\hat{s}^{(l)}=s^{(l)}$ outside of a slightly bigger neighborhood $Op(K) \supset Op\big(F_1(K)\big)$.
 \end{itemize}
 \end{theorem}
 
 \begin{remark}\label{remarks on l-holonomic approximation} $ $
 \begin{enumerate}
 \item Every section $s:M \to X^{(r)}$ of $p^r$ is $0$-holonomic. Therefore, taking $l=0$ in Theorem \ref{Approximation of $l$-holonomic sections} we recover the classical Theorem \ref{classical holonomic approximation}.
\item Theorem \ref{Approximation of $l$-holonomic sections} also holds in relative form. To be precise, if $s$ is already holonomic on $Op(A)$ for some closed subset $A \subset M$, then we can arrange it so that $F_t=id_M$ and $\hat{s}=s$ on $Op(A)$. The same comment applies to Theorem \ref{Approximation of perp-holonomic sections} below.
\end{enumerate}
 \end{remark}
Our second refinement concerns a specific class of $(r-1)$-holonomic sections of $p^r$, which we call $\perp$-holonomic.  Informally, we can describe a $\perp$-holonomic section as a section of $p^r$ which differs from a holonomic section only by the formal analogue of a pure order $r$ partial derivative.  We show that when the holonomic approximation process is applied to a $\perp$-holonomic section it is not only possible to globally control the $(r-1)$-jet component, but it is also possible to globally control the order $r$ information complementary to this formal pure order $r$ partial derivative. The precise statement is most cleanly phrased in terms of the bundle $X^{\perp}$, which we define below. This bundle was first introduced in by Gromov in  \cite{G86} in the context of convex integration, where the language of $\perp$-jets is used to construct iterated convex hull extensions of partial differential relations. A thorough exposition of the theory of convex integration, including details on the geometry of $X^{\perp}$, can be found in Spring's book \cite{S98}. 

 Let $\tau \subset TM$ be a hyperplane field on $M$. We associate to $\tau$ a bundle $p^{\perp}:X^{\perp} \to M$ in the following way. The fibre of $p^{\perp}$ over a point $x \in M$ consists of equivalence classes of germs of sections $h:Op(x) \to X$ of $p$, where two germs are identified if their $(r-1)$-jets $y=j^{r-1}(h)(x) \in X^{(r-1)}$ at the point $x$ are the same and moreover if the restriction of their tangent maps $d\big (j^{r-1}(h) \big) :T_xM \to T_{y}X^{(r-1)}$ to the hyperplane $\tau_x \subset T_xM$ are the same. When $X=M \times N$ is a trivial bundle we also write $X^\perp= J^\perp(M,N)$. Observe that $X^{\perp}$ lies between $X^{(r)}$ and $X^{(r-1)}$ in the sense that there exist natural affine bundle structures $p^{r}_\perp:X^{(r)} \to X^{\perp}$ and $p^{\perp}_{r-1}:X^{\perp} \to X^{(r-1)}$ such that $p^{\perp}_{r-1} \circ p^r_{\perp}=p^r_{r-1}$. Given a section $s:A \to X^{(r)}$ of $p^r$ defined over any subset $A \subset M$, we call $s^{\perp}=p^r_\perp \circ s : A \to X^{\perp}$ the $\perp$-component of $s$. Given a section $h:U \to X$ of $p$ defined over an open subset $U \subset M$, we denote by $j^{\perp}(h):U \to X^{\perp}$ the section of $p^{\perp}$ formed by the $\perp$-jets of $h$. Explicitly, $j^\perp(h)=p^r_\perp \circ j^{r}(h)$. Sections of $p^{\perp}$ of the form $j^{\perp}(h)$ are called holonomic.  
 
  \begin{definition} A section $s:U \to X^{(r)}$ of $p^r$ defined over an open subset $U \subset M$ is called \mbox{$\perp$-holonomic} with respect to a hyperplane field $\tau \subset TM$ if $s^{\perp}$ is a holonomic section of $p^{\perp}$. 
 \end{definition}
 
\begin{figure}[h]
\includegraphics[scale=0.6]{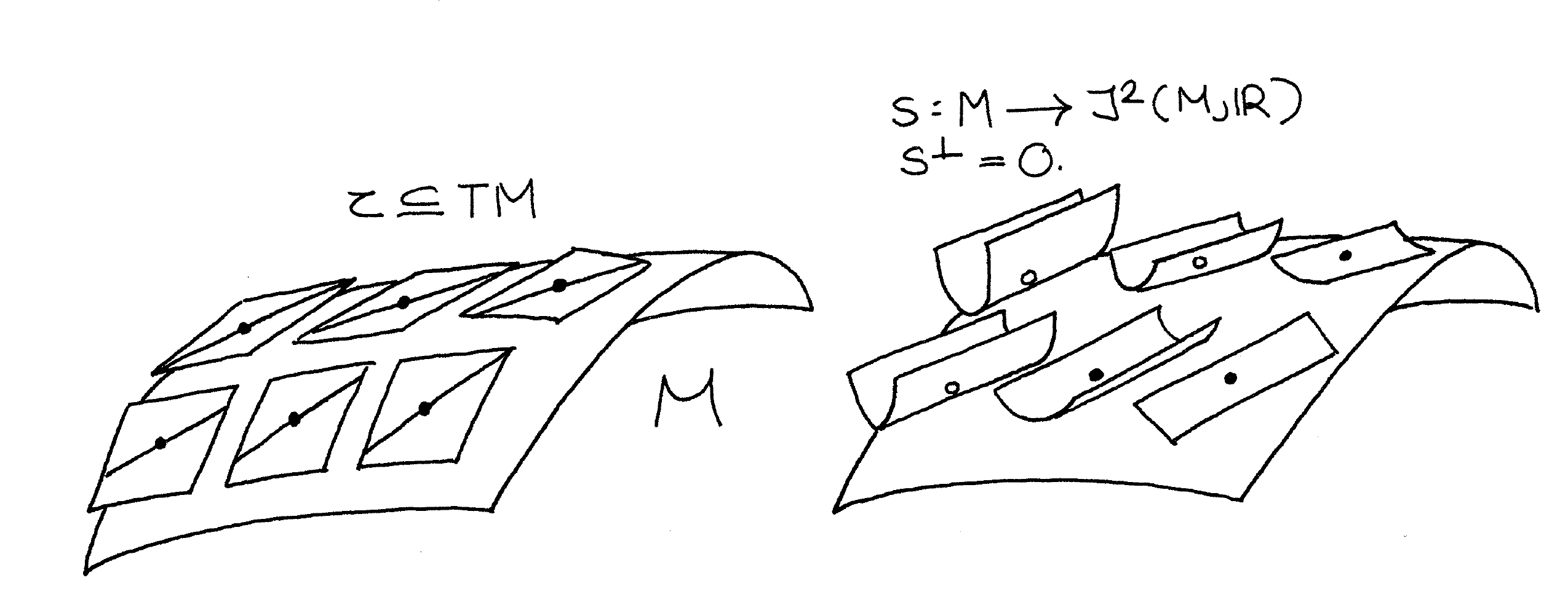}
\caption{The section $s$ of the $2$-jet bundle $J^2(M, \bR)$ is $\perp$-holonomic with respect to the hyperplane field $\tau \subset TM$. Indeed, $s^\perp=0$ is a holonomic section of $J^\perp(M, \bR)$.}
\label{primitieve}
\end{figure}

 \begin{theorem}[holonomic approximation lemma for $\perp$-holonomic sections]\label{Approximation of perp-holonomic sections}
 Let $s:M \to X^{(r)}$ be a section of the $r$-jet bundle of a fibre bundle $p:X \to M$ and let $K \subset M$ be a polyhedron of positive codimension. Suppose that the section $s$ is $\perp$-holonomic with respect to some hyperplane field $\tau \subset TM$. Then there exists an isotopy $F_t:M \to M$ and a holonomic section $\hat{s}:M \to X^{(r)}$ such that the following properties hold.
 \begin{itemize}
 \item $\hat{s}$ is $C^0$-close to $s$ on $Op\big(F_1(K) \big)$. 
 \item $\hat{s}^{\perp}$ is $C^0$-close to $s^{\perp}$ on all of $M$.
 \item $F_t$ is $C^0$-small.
 \item $F_t=id_M$ and $\hat{s}^{\perp}=s^{\perp}$ outside of a slightly bigger neighborhood $Op(K) \supset Op\big( F_1(K) \big)$.
 
 \end{itemize}
 \end{theorem}

 \begin{remark}
 In fact, the proof of Theorem \ref{Approximation of perp-holonomic sections} will produce a very specific isotopy $F_t$. Informally, we can say that $F_t$ wiggles $K$ in such a way that the wiggles are parallel to the hyperplane field $\tau$. More formally, we can arrange so that the pulled back hyperplane field $F_t^* \tau $ is $C^0$-close to $\tau $ for all $t \in [0,1]$, see Figure \ref{perp-holonomic approximation}. An analogous comment applies in the parametric version Theorem \ref{Parametric approximation of perp-holonomic sections} below.
 \end{remark}
 
\begin{figure}[h]
\includegraphics[scale=0.6]{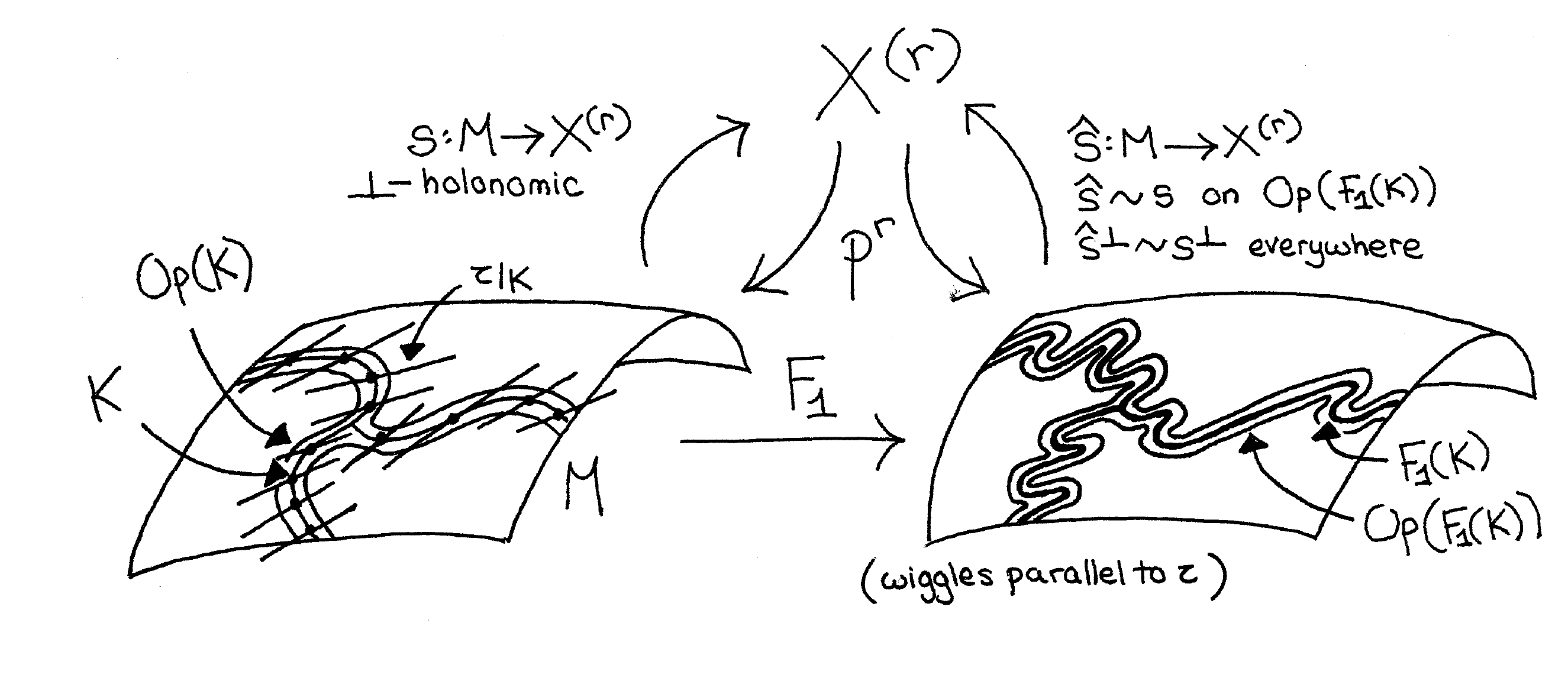}
\caption{The wiggles of the deformed subset $F_1(K)$ are parallel to $\tau \subset TM$ }
\label{perp-holonomic approximation}
\end{figure}

  \subsection{Parametric versions}\label{Parametric versions}
Our main results Theorem \ref{Approximation of $l$-holonomic sections}  and Theorem \ref{Approximation of perp-holonomic sections} remain true in families. We take our parameter space to be a compact manifold $Z$, whose boundary $\partial Z$ is possibly nonempty. We consider families of sections parametrized by $Z$. For example, a family of sections $s_z:M \to X^{(r)}$ depending on the parameter $z \in Z$ is a smooth mapping $s:Z \times M\to X^{(r)}$ such that for every $z \in Z$ the assignment $x \mapsto s(z,x)$ defines a smooth section $s_z$ of $p^r$. Additionally, we allow the polyhedron $K$ to vary with the parameter in the following way.
\begin{definition} A closed subset $K \subset Z \times M$ is called a fibered polyhedron if it is a subcomplex of a smooth triangulation of $Z \times M$ which is in general position with respect to the fibres $z \times M$, $z \in Z$.
\end{definition}

 A consequence of this definition is that for every $z \in Z$ the subset $K_z \subset M$ given by $K \cap (z \times M)= z  \times K_z$ is a polyhedron in $M$, see Figure \ref{fibered}. If $K$ has positive codimension in $Z \times M$, then $K_z$ has positive codimension in $M$ for all $z \in Z$. We are now ready to formulate the parametric analogues of Theorems \ref{Approximation of $l$-holonomic sections} and \ref{Approximation of perp-holonomic sections}.
  \begin{theorem}[parametric holonomic approximation lemma for $l$-holonomic sections]\label{Parametric approximation of $l$-holonomic sections}
 Let $s_z:M \to X^{(r)}$ be a family of sections of the $r$-jet bundle associated to a fibre bundle $p:X \to M$ parametrized by a compact manifold $Z$. Let $K \subset Z \times M$ be a fibered polyhedron of positive codimension. Suppose that for some $l<r$ the sections $s_z$ are $l$-holonomic for all $z \in Z$ and that they are holonomic for $z \in Op(\partial Z)$. Then there exists a family of isotopies $F^z_t:M \to M$ and a family of holonomic sections $\hat{s}_z:M \to X^{(r)}$ such that the following properties hold.
 \begin{itemize}
 \item $\hat{s}_z$ is $C^0$-close to $s_z$ on $Op\big(F^z_1(K_z)\big)$.
 \item $\hat{s}_z^{(l)}$ is $C^0$-close to $s_z^{(l)}$ on all of $M$.
 \item $F^z_t$ is $C^0$-small.
 \item $F^z_t=id_M$ and $\hat{s}^{(l)}_z=s^{(l)}_z$ outside of a slightly bigger neighborhood $Op(K_z) \supset Op\big(F^z_1(K_z)\big)$.
 \item $F^z_t=id_M$ and $\hat{s}_z=s_z$ for $z \in Op(\partial Z)$.
 \end{itemize}
 \end{theorem}
   \begin{theorem}[parametric holonomic approximation lemma for $\perp$-holonomic sections]\label{Parametric approximation of perp-holonomic sections}
 Let $s_z:M \to X^{(r)}$ be a family of sections of the $r$-jet bundle associated to a fibre bundle $p:X \to M$ parametrized by a compact manifold $Z$. Let $K \subset Z \times M$ be a fibered polyhedron of positive codimension. Suppose that the sections $s_z$ are $\perp$-holonomic with respect to some family of hyperplane fields $\tau_z \subset TM$ for all $z \in Z$ and that they are holonomic for $z \in Op(\partial Z)$. Then there exists a family of isotopies $F^z_t:M \to M$ and a family of holonomic sections $\hat{s}_z:M \to X$ such that the following properties hold.
 \begin{itemize}
 \item $\hat{s}_z$ is $C^0$-close to $s_z$ on $Op\big(F^z_1(K_z)\big)$.
 \item $\hat{s}_z^\perp$ is $C^0$-close to $s_z^{\perp}$ on all of $M$.
 \item $F^z_t$ is $C^0$-small.
 \item $F^z_t=id_M$ and $\hat{s}_z^\perp=s_z^\perp$ outside of a slightly bigger neighborhood $Op(K_z) \supset Op\big(F^z_1(K_z)\big)$.
 \item $F^z_t=id_M$ and $\hat{s}_z=s_z$ for $z \in Op(\partial Z)$.
 \end{itemize}
 \end{theorem}
 \begin{figure}[h]
\includegraphics[scale=0.6]{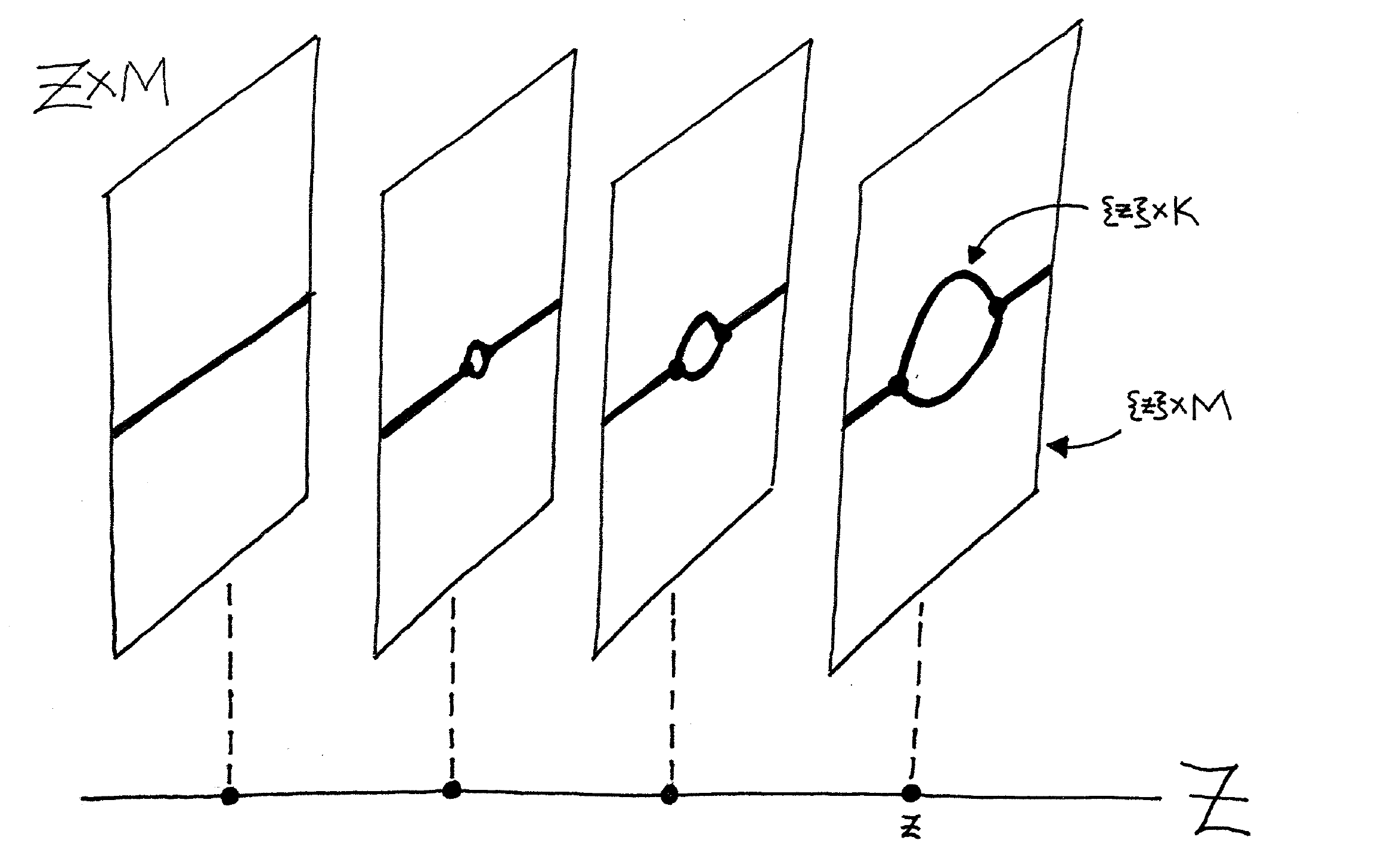}
\caption{A typical fibered polyhedron $K \subset Z \times M$.}
\label{fibered}
\end{figure}

 \begin{remark} $ $
 \begin{enumerate}
\item Observe that the formulation of the parametric Theorems \ref{Parametric approximation of $l$-holonomic sections} and \ref{Parametric approximation of perp-holonomic sections} is relative with respect to the closed subset $\partial Z$ of the parameter space $Z$. In typical applications, taking $(Z, \partial Z)=(D^m,S^{m-1})$ leads to a result about certain relative homotopy groups vanishing, which can be rephrased in terms of the existence of a homotopy equivalence (the $h$-principle).
\item Theorems \ref{Parametric approximation of $l$-holonomic sections} and \ref{Parametric approximation of perp-holonomic sections} also hold in relative form with respect to a closed subset $A \subset M$. The statement is analogous to the one phrased in Remark \ref{remarks on l-holonomic approximation}.
\item Note that $s_z^{\perp}$ is a section of the bundle $X^{\perp}_{z} \to M$ associated to the hyperplane field $\tau_z \subset TI^m$, which varies with $z \in Z$. We can view the collection $s_z^{\perp}$ as a single section of the bundle $X_Z^{\perp} \to Z \times M$ whose fibre over $(z,m)$ is $X^{\perp}_z$.
\end{enumerate}
\end{remark}

\subsection{Applications to symplectic and contact topology}\label{Applications to symplectic topology}
We begin with an example which illustrates the main point. Suppose that $f:L^n \to W^{2n}$ is a Lagrangian embedding of a manifold $L$ into a symplectic manifold $(W, \omega)$. Denote by $\pi:\Lambda_n(W) \to W$ the Grassmannian bundle of Lagrangian planes in $TW$. The fibre of $\pi$ over a point $x \in W$ consists of the linear Lagrangian subspaces of the symplectic vector space $(T_xW, \omega_x)$. The Gauss map $G(df):L \to \Lambda_n(W)$ of the embedding $f$ is defined by $G(df)(q)=df(T_qM) \subset T_{f(q)}W$.
\begin{definition} A tangential rotation of $f$ is a compactly supported deformation $G_t:L \to \Lambda_n(W)$, $t \in [0,1]$, of the Gauss map $G_0=G(df)$ such that $\pi \circ G_t=f$. 
\end{definition}

\begin{figure}[h]
\includegraphics[scale=0.6]{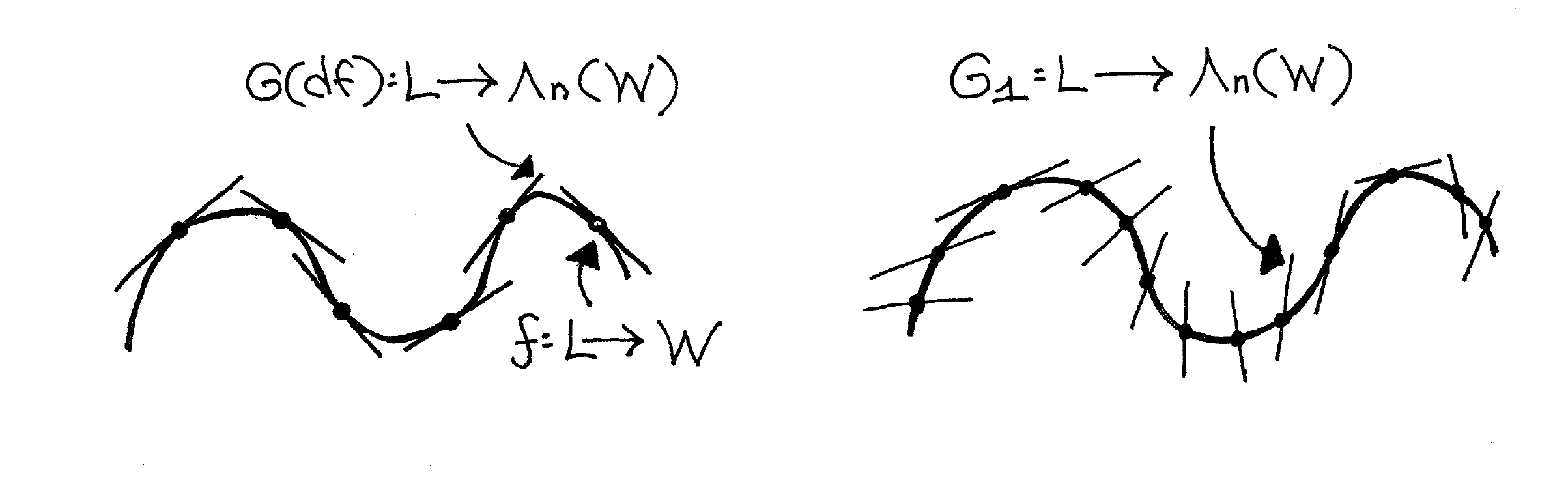}
\caption{A tangential rotation $G_t$ of a Lagrangian embedding $f$.}
\label{rotation}
\end{figure}

By compactly supported we mean that $G_t=G(df)$ for all $t \in [0,1]$ outside of a compact subset of $L$. The following approximation result is a simple corollary of our holonomic approximation lemma for $l$-holonomic sections.

\begin{theorem}\label{tangential rotation}
Let $K \subset L$ be a polyhedron of positive codimension and let $G_t:L \to \Lambda_n(M)$ be a tangential rotation of a Lagrangian embedding $f:L \to M$. Then there exists a compactly supported ambient Hamiltonian isotopy $\varphi_t:W \to W$ such that $G\big(d(\varphi_t \circ f) \big)$ is $C^0$-close to $G_t$ on $Op(K)$.
\end{theorem}

\begin{remark}\label{remark to tangential rotation} $ $
\begin{enumerate}
\item We can take $\varphi_t$ to be $C^0$-close to the identity $id_M$ on all of $M$.
\item Moreover, we can also arrange it so that $\varphi_t=id_M$ outside of an arbitrarily small neighborhood of $f(L)$ in $M$.
\item The statement holds in relative form. Namely, if $G_t=G(df)$ on $Op(A) \subset L$ for some closed subset $A \subset L$, then we can take $\varphi_t$ so that $\varphi_t=id_M$ on $Op \big( f(A) \big) \subset M$.
\item An analogous approximation result holds for tangential rotations of Legendrian embeddings into a contact manifold $(W^{2n+1}, \xi)$. In this case $\pi:\Lambda_n(W) \to W$ is the Grassmannian bundle whose fibre over a point $x \in W$ consists of the linear Lagrangian subspaces of the symplectic vector space $(\xi^{2n}_x, d \alpha_x)$, where $\xi=\ker(\alpha)$ near the point $x$.

\end{enumerate}
\end{remark}
\begin{proof}
Let $\cN \subset W$ be a Weinstein neighborhood of $f(L)$. This means that $\cN$ is a tubular neighborhood of $f(L)$ in $W$ which is symplectomorphic to a neighborhood of the zero section in the cotangent bundle $(T^*L, dp \wedge dq)$. If we fix a Riemannian metric on $L$ we can choose $\cN \simeq T^*_\delta L$ for some $\delta>0$, where $T^*_\delta L$ consists of the cotangent vectors $p \in T^*L$ such that $|| p ||< \delta$. 

For small time $t$ we can think of the tangential rotation $G_t$ as a family of sections $s_t: L \to J^2(L, \bR)$ such that $s_0=0$ and $s_t^{(1)}=0$. In fact, by first subdividing the time interval finely enough, we can reduce to the case where $s_t$ is defined for all $t \in [0,1]$. The point here is that the planes $G_t(q) \subset T_{f(q)}W \simeq T_{(q,0)}(T^*L)$ must remain graphical over $T_qL \subset T_{(q,0)}(T^*L)$. 

The parametric version of the holonomic approximation lemma for $l$-holonomic sections can be applied to produce an isotopy $F_t:L \to L$ and a family of functions $h_t:L\to \bR$, $h_0=0$,  such that $j^{2}(h_t)$ is $C^0$-close to $s_t$ on $Op\big(F_t(K) \big)$ and such that $j^1(h_t)$ is $C^0$-small on all of $L$. In particular, we may assume that $|| dh_t ||< \delta$ for all $t \in [0,1]$. Hence we can think of the composition $f_t=dh_t \circ F_t:L \to T^*_\delta L$ as a compactly supported exact homotopy of Lagrangian embeddings $f_t:L \to W$. Every such homotopy is induced by a compactly supported ambient Hamiltonian isotopy $\varphi_t:M \to M$ satisfying the required properties. \end{proof}

If we attempt to prove Theorem \ref{tangential rotation} using the classical holonomic approximation lemma \ref{classical holonomic approximation} instead, we run into the following difficulty. The functions $h_t$ produced by the holonomic approximation would a priori only be defined in open subsets $Op\big(F_t(K) \big) \subset L$. We would therefore need to extend $h_t$ to the whole of $L$ by hand. The most straightforward way of doing so is to choose a family of cutoff functions $\psi_t:L \to \bR$ supported on the domain of $h_t$ such that $\psi_t=1$ near $F_t(K)$. The product $\widetilde{h}_t= \psi_t \cdot h_t :L \to \bR$ is then well defined on all of $L$. It follows that the composition $f_t=d\widetilde{h}_t \circ F_t :L \to T^*L$ is an exact homotopy of Lagrangian embeddings whose Gauss map provides the desired approximation near $K$.

 If $|| d \widetilde{h}_t ||< \delta$ on all of $L$, then $f_t$ has image contained in $T^*_\delta L$ and we can think of $f_t$ as an exact homotopy of Lagrangian embeddings $f_t:L \to W$ as before. Observe, however, that there is no guarantee that $d\widetilde{h}_t=\psi_t dh_t + h_t d \psi_t$ has norm $< \delta$, because $ \psi_t$ will typically have a very large derivative. Indeed, the wiggling $F_t$ is quite dramatic, see Figure \ref{holonomic failure} for an illustration. Therefore, $f_t(L)$ might escape outside of our Weinstein neighborhood $T^*_\delta L \simeq \cN$. Hence $f_t$ does not correspond to an exact homotopy of Lagrangian embeddings into $W$ and our proof breaks down. Our holonomic approximation lemma for $l$-holonomic sections with $l=1$ precisely provides the necessary global control on the $1$-jet component so that this issue does not arise.
 \begin{figure}[h]
\includegraphics[scale=0.55]{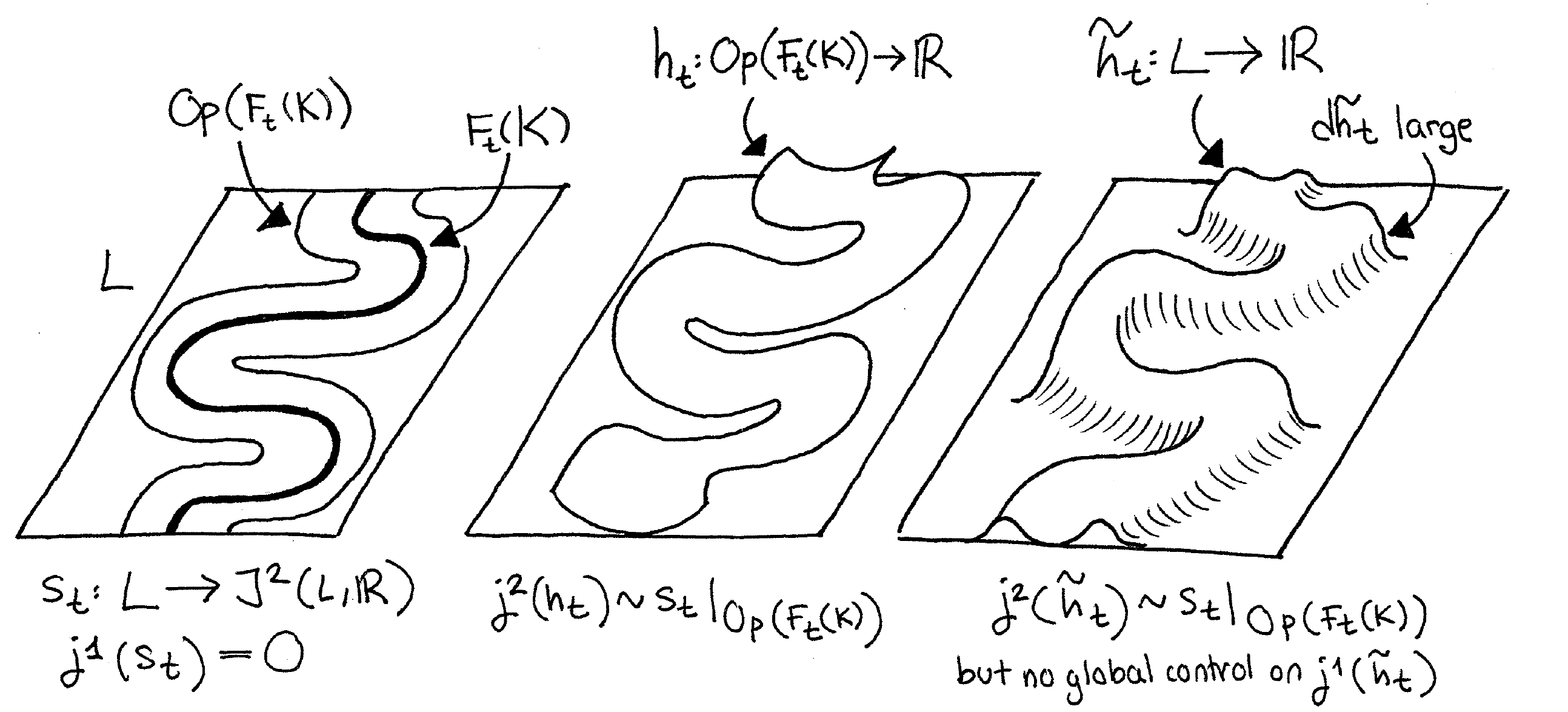}
\caption{The potential problem with the cutting off of $h_t$ after a naive application of the classical holonomic approximation lemma.}
\label{holonomic failure}
\end{figure}

 The parametric version of Theorem \ref{tangential rotation} also holds and is proved in the same way.
 
 \begin{theorem}\label{parametric tangential rotation}
Let $K \subset Z \times L$ be a fibered polyhedron of positive codimension and let $G^z_t:L \to \Lambda_n(M)$ be a family of tangential rotations of Lagrangian embeddings $f^z:L \to M$ parametrized by a compact manifold $Z$ such that $G^z_t=G(df^z)$ for $z \in Op(\partial Z)$. Then there exists a family of compactly supported ambient Hamiltonian isotopies $\varphi^z_t:W \to W$ such that $G\big(d(\varphi^z_t \circ f^z) \big)$ is $C^0$-close to $G^z_t$ on $Op(K)$ and such that $\varphi^z_t=id_M$ for $z \in Op(\partial Z)$.
\end{theorem}

\begin{remark}
Analogous observations to the ones made in Remark \ref{remark to tangential rotation} apply.
\end{remark}

The following $h$-principle for directed embeddings follows immediately from the above approximation results. First, we recall the following definition of Gromov.  
\begin{definition} Given subsets $D \subset \Lambda_n(W)$ and $S \subset L$, we say that a Lagrangian embedding $f:L \to W$ is $D$-directed along $S$ if $G(df)(S) \subset D$.
\end{definition}

\begin{theorem}\label{A-directed}
Let $f:L \to W$ be a Lagrangian embedding, let $K \subset L$ be a polyhedron of positive codimension and let $D \subset \Lambda_n(W)$ be an open subset. Suppose that there exists a tangential rotation $G_t:L \to \Lambda_n(W)$ of $f$ such that $G_1(K) \subset D$. Then there exists a compactly supported Hamiltonian isotopy $\varphi_t:W \to W$ such that $\varphi_1 \circ f$ is $D$-directed along $Op(K)$. 
\end{theorem}
\begin{remark} $ \, $
\begin{enumerate}
\item This $h$-principle also holds in $C^0$-close, relative and parametric versions. We leave it to the reader to formulate the appropriate statements.
\item We can choose the Hamiltonian isotopy so that $\varphi_t=id_W$ outside of an arbitrarily small neighborhood of $f(K)$ in $W$.
\item An analogous $h$-principle holds for Legendrian embeddings into a contact manifold $(W^{2n+1}, \xi)$ that are $D$-directed along a polyhedron of positive codimension. 
\end{enumerate}

\end{remark}

Analogous problems in geometric topology have been studied by several authors. In \cite{G86}, Gromov proved an $h$-principle for $D$-directed smooth embeddings of an open manifold into some ambient manifold which holds for any open subset $D$ of the Grassmannian of the ambient manifold. See \cite{S00} for a discussion of Gromov's argument. Rourke and Sanderson gave two independent proofs of this result in \cite{RS97} and \cite{RS00}. Another proof was obtained by Eliashberg and Mishachev using their holonomic approximation lemma \cite{EM01}. For embeddings of a closed manifold one cannot hope to prove an $h$-principle for $D$-directed embeddings when $D$ is an arbitrary open subset of the Grassmannian. However, for certain special subsets $D$, called ample, Gromov proved in \cite{G73} and \cite{G86} that an $h$-principle does hold. In a different direction, Eliashberg and Mishachev showed in \cite{EM09} that an $h$-principle for $D$-directed embeddings of a closed manifold hold for an arbitrary open $D$, but provided that we relax the notion of an embedding to that of a wrinkled embedding. In \cite{AG15} we prove a symplectic and contact analogue of this last result using the tools developed in the present paper.
 
 \begin{example}
 Let $\eta \subset TW$ be a distribution of $k$-planes in a symplectic or contact manifold $W$. Set $D_x=\{ P_x \in \Lambda_n(W)_x: \, \, P_x \pitchfork \eta_x \}$ for each $x \in W$. Then $D=\bigcup_{x\in W} D_x$ is an open subset of $\Lambda_n(W)$. We obtain a full $h$-principle for Lagrangian or Legendrian embeddings which are transverse to an ambient distribution near a given subset of positive codimension. In the particular case where $\eta=\ker(d \rho)$ for $\rho:W \to B$ a Lagrangian or Legendrian fibration, we can rephrase the result as an $h$-principle for Lagrangian or Legendrian embeddings whose front is nonsingular along a given subset of positive codimension.
 \end{example}
 
The main application (and source of motivation) for our holonomic approximation lemma for \mbox{$\perp$-holonomic sections}, as well as its parametric version, is given in \cite{AG15}. The control on the $\perp$-jet component is a key ingredient in the proof of the $h$-principle for the (global) simplification of singularities of Lagrangian and Legendrian fronts. Indeed, when attempting to apply a holonomic approximation argument near the $(n-1)$-skeleton of a Lagrangian or Legendrian submanifold, difficulties similar to the one illustrated in the proof of Theorem \ref{tangential rotation} above inevitably arise. The situation is in fact much more subtle because we need to respect a certain decomposition of a tangential rotation into so-called simple tangential rotations. Theorems \ref{Approximation of perp-holonomic sections} and \ref{Parametric approximation of perp-holonomic sections} provide the precise control needed to make the proof work. 
 
\subsection{Idea of the proof}

The strategy of proof is to carry out a sequence of reductions which simplify our refined holonomic approximation lemmas for $l$- and $\perp$-holonomic sections to a problem described by a concrete local model. We can then carefully keep track of the geometry behind the holonomic approximation process in this carefully chosen model and establish the necessary estimates to achieve the desired global control. The outline of the paper is roughly as follows. 

In Section \ref{Localization of the problem} we reduce our global results to the local relative statements corresponding to the jet space $J^r(\bR^m,\bR^n)$ over the unit cube $I^m=[-1,1]^m$. In Section \ref{Geometry of jet spaces} we study the space $J^r(\bR^m,\bR^n)$ and reduce the holonomic approximation lemma for $l$-holonomic sections to the holonomic approximation lemma for $\perp$-holonomic sections. For a section which is $\perp$-holonomic with respect to a hyperplane field $\tau$, we construct in Section \ref{holonomic approximation with controlled cutoff} a holonomic approximation with controlled $\perp$-component by wiggling the polyhedron $K$ in a way such that the wiggles are parallel to the hyperplanes in $\tau$. However, we cannot implement such a wiggling near the region where $\tau$ is almost tangent to $K$. A preliminary adjustment is therefore necessary in this region. This adjustment is performed in Section \ref{transversality adjustment}.  

We should note that it is possible to prove the holonomic approximation lemma for $l$-holonomic sections in a more direct manner. One can extend by hand the holonomic approximation formulas written down in \cite{EM02} or appeal to abstract extension results to reach the desired conclusion. However, we choose to deduce the holonomic approximation lemma for $l$-holonomic sections as a corollary of the holonomic approximation lemma for $\perp$-holonomic sections. 

The main reason for doing so is that such a reduction involves decompositions of $(r-1)$-holonomic sections into so-called primitive sections, which we define in Section \ref{Holonomic trivialization}. A similar strategy appears in Gromov's work \cite{G86} and is developed further in Spring's book \cite{S98}, related to the construction of iterated convex hull extensions in the theory of convex integration. Moreover, primitive sections play a crucial role the proof of our $h$-principle for the simplification of singularities of Lagrangian and Legendrian fronts  \cite{AG15}, where they correspond to a particularly simple type of tangential rotation. We hope that the general idea of working one pure partial derivative at a time may have further applications to the philosophy of the $h$-principle and so have attempted to present the elements of the strategy as clearly as possible.

\subsection{Acknowledgements}
I am very grateful to my advisor Yasha Eliashberg for many insightful conversations and to Fran Presas from whom I first learnt about the philosophy of the $h$-principle. I am indebted to the ANR Microlocal group who held a workshop in January 2017 to dissect an early version of the paper and in particular to H\'el\`ene Eynard-Bontemps for spotting various errors and making useful suggestions for fixing them. I am also thankful to Nikolai Mishachev for reading the first draft of this paper. 

\section{Localization of the problem}\label{Localization of the problem}
\subsection{Holonomic trivialization}\label{Holonomic trivialization}
For a general fibre bundle $p:X^{m+n} \to M^m$, the bundle of $r$-jets $p^r:X^{(r)}\to M$ can be messy to work with globally. However, global $h$-principle type problems can often be reduced to a local relative statement. In this section we explain how this reduction is accomplished for our refined holonomic approximation lemmas. We choose to work with the unit cube $I^m=[-1,1]^m \subset \bR^m$ as our local model. In what follows we use the language of \, $l$- and $\perp$-holonomic sections introduced in Section \ref{Improved holonomic approximation}. We begin by recalling from \cite{EM02} the following simple but crucial observation.

\begin{remark}[holonomic trivialization] \label{holonomic trivialization} Let $\hat{s}:M \to X^{(r)}$ be a holonomic section of $p^r$ and let $Q \subset M$ be an embedded cube $Q \simeq I^m$. Then there exists a neighborhood $\cN \subset X^{(r)}$ of the image Im$(\hat{s})$ such that $(p^r)^{-1}(Q) \cap \cN \simeq J^r(\bR^m, \bR^n)|_{I^m}$.
\end{remark} 
\begin{proof} Since the section $\hat{s}$ is holonomic, we have $\hat{s}=j^r(h)$ for some section $h:M \to X$ of $p$. Observe that the fibration $p:X \to M$ is trivial over the contractible subset $Q$. Hence a neighborhood of the image $h(Q)\subset X$ in $p^{-1}(Q)$ is diffeomorphic to $Q \times \bR^n$. It follows that a neighborhood of the image $\hat{s}(Q) \subset X^{(r)}$ in $(p^r)^{-1}(Q)$ is diffeomorphic to $J^r(Q, \bR^n) \simeq J^r(\bR^m, \bR^n)|_{I^m}$. \end{proof}

Under the above identification, sections $s : M \to \cN \subset X^{(r)}$ such that $s=\hat{s}$ on $Op\big(M \setminus \text{int}(Q)\big)$ correspond to sections $\sigma:I^m \to J^r(\bR^m, \bR^n)$ such that $\sigma=0$ on $Op(\partial I^m)$. The section $s$ is holonomic if and only if the section $\sigma$ is holonomic. The section $\hat{s}$ itself corresponds to the zero section $\sigma=0$. 

Similarly, $l$-holonomic sections $s:M \to  \cN \subset X^{(r)}$ such that $s=\hat{s}$ on $Op\big(M \setminus \text{int}(Q)\big)$ and such that $s^{(l)}=\hat{s}^{(l)}$ on all of $M$ correspond to sections $\sigma : I^m \to J^r(\bR^m, \bR^n)$ such that $\sigma=0$ on $Op(\partial I^m)$ and such that $\sigma^{(l)}=0$ on all of $I^m$. 

Fix a hyperplane field $\tau \subset TM$. Then $\perp$-holonomic sections $s:M \to \cN \subset X^{(r)}$ such that $s=\hat{s}$ on $Op\big(M \setminus \text{int}(Q)\big)$ and such that $s^{\perp}=\hat{s}^{\perp}$ on all of $M$ (with respect to $\tau$) correspond to sections $\sigma:I^m \to J^r(\bR^m, \bR^n)$ such that $\sigma=0$ on $Op(\partial I^m)$ and such that $\sigma^{\perp}=0$ on all of $I^m$ (with respect to the hyperplane field associated to $\tau$ under the identification $Q \simeq I^m$). This last remark motivates the following definition, which we will use repeatedly in what follows.
\begin{definition}\label{primitive definition} A section $\sigma: I^m \to J^r(\bR^m, \bR^n)$ is called primitive with respect to a hyperplane field $\tau \subset TI^m$ if $\sigma^{\perp}=0$.
\end{definition}

\subsection{The local relative statements}\label{The local relative statements}
The global versus local dictionary described in the previous subsection leads us to formulate the following local relative versions of our main results. We first state the non-parametric versions.
\begin{theorem}[localized holonomic approximation lemma for $l$-holonomic sections]\label{localized holonomic approximation lemma for $l$-holonomic sections}
Fix $k<m$. Let $\sigma:I^m \to J^r(\bR^m, \bR^n)$ be a section such that the following properties hold.
\begin{itemize}
\item $\sigma=0$ on $Op(\partial I^m)$.
\item $\sigma^{(l)}=0$ on all of $I^m$ for some $l<r$.
\end{itemize}
Then there exists an isotopy $F_t:I^m \to I^m$ and a holonomic section $\hat{\sigma} :I^m \to J^r(\bR^m, \bR^n)$ such that the following properties hold.
\begin{itemize}
\item $\hat{\sigma}$ is $C^0$-close to $\sigma$ on $Op\big( F_1(I^k) \big)$.
\item $\hat{\sigma}^{(l)}$ is $C^0$-small on all of $I^m$.
\item $F_t$ is $C^0$-small.
\item $F_t=id_{I^m}$ and $\hat{\sigma}=0$ on $Op(\partial I^m)$.
\end{itemize}
\end{theorem}

\begin{theorem}[localized holonomic approximation lemma for $\perp$-holonomic sections]\label{localized holonomic approximation lemma for perp-holonomic sections}
Fix $k<m$. Let $\sigma:I^m \to J^r(\bR^m, \bR^n)$ be a section such that the following properties hold.
\begin{itemize}
\item $\sigma=0$ on $Op(\partial I^m)$.
\item $\sigma^{\perp}=0$ on all of $I^m$ with respect to some hyperplane field $\tau \subset TI^m$.
\end{itemize}
Then there exists an isotopy $F_t:I^m \to I^m$ and a holonomic section $\hat{\sigma} :I^m \to J^r(\bR^m, \bR^n)$ such that the following properties hold.
\begin{itemize}
\item $\hat{\sigma}$ is $C^0$-close to $\sigma$ on $Op\big( F_1(I^k) \big)$.
\item $\hat{\sigma}^{\perp}$ is $C^0$-small on all of $I^m$.
\item $F_t$ is $C^0$-small.
\item $F_t=id_{I^m}$ and $\hat{\sigma}=0$ on $Op(\partial I^m)$.
\end{itemize}
\end{theorem}

The global holonomic approximation lemmas for $l$- and $\perp$-holonomic sections \ref{Approximation of $l$-holonomic sections} and \ref{Approximation of perp-holonomic sections} follow from the local relative statements \ref{localized holonomic approximation lemma for $l$-holonomic sections} and \ref{localized holonomic approximation lemma for perp-holonomic sections} by induction over the skeleton of the polyhedron $K$, working one cube at a time. At each step we use the holonomic trivialization \ref{holonomic trivialization} to reduce the global problem to a local problem. Observe that the relative versions of the global holonomic approximation lemmas for $l$- and $\perp$-holonomic sections (see Remark \ref{remarks on l-holonomic approximation}) also follow from the above local relative statements. 

Similarly, the parametric global holonomic approximation lemmas \ref{Parametric approximation of $l$-holonomic sections} and \ref{Parametric approximation of perp-holonomic sections}, including the corresponding relative versions, follow from the parametric local relative statements phrased below. In this case we also localize with respect to the parameter space, setting $Z=I^q$, the unit $q$-dimensional cube $[-1,1]^q$. 

\begin{theorem}[parametric localized holonomic approximation lemma for $l$-holonomic sections]\label{parametric localized holonomic approximation lemma for $l$-holonomic sections}
Fix $k<m$. Let $\sigma_z:I^m \to J^r(\bR^m, \bR^n)$ be a family of sections parametrized by the unit cube $I^q$ such that the following properties hold.
\begin{itemize}
\item $\sigma_z=0$ on $Op(\partial I^m)$.
\item $\sigma^{(l)}_z=0$ on all of $I^m$ for some $l<r$.
\item $\sigma_z=0$ on all of $I^m$ for $z \in Op(\partial I^q)$.
\end{itemize}
Then there exists a family of isotopies $F^z_t:I^m \to I^m$ and a family of holonomic sections $\hat{\sigma}_z :I^m \to J^r(\bR^m, \bR^n)$ such that the following properties hold.
\begin{itemize}
\item $\hat{\sigma}_z$ is $C^0$-close to $\sigma_z$ on $Op\big( F^z_1(I^k) \big)$.
\item $\hat{\sigma}^{(l)}_z$ is $C^0$-small on all of $I^m$.
\item $F^z_t$ is $C^0$-small.
\item $F^z_t=id_{I^m}$ and $\hat{\sigma}_z=0$ on $Op(\partial I^m)$.
\item $F^z_t=id_{I^m}$ and $\hat{\sigma}_z=0$ on all of $I^m$ for $z \in Op(\partial I^q)$.
\end{itemize}
\end{theorem}

\begin{theorem}[parametric localized holonomic approximation lemma for $\perp$-holonomic sections]\label{parametric localized holonomic approximation lemma for perp-holonomic sections}
Fix $k<m$. Let $\sigma_z:I^m \to J^r(\bR^m, \bR^n)$ be a family of sections parametrized by the unit cube $I^q$ such that the following properties hold.
\begin{itemize}
\item $\sigma_z=0$ on $Op(\partial I^m)$.
\item $\sigma^{\perp}_z=0$ on all of $I^m$ with respect to some family of hyperplane fields $\tau_z \subset TI^m$.
\item $\sigma_z=0$ on all of $I^m$ for $z \in Op(\partial I^q)$.
\end{itemize}
Then there exists a family of isotopies $F^z_t:I^m \to I^m$ and a family of holonomic sections $\hat{\sigma}_z :I^m \to J^r(\bR^m, \bR^n)$ such that the following properties hold.
\begin{itemize}
\item $\hat{\sigma}_z$ is $C^0$-close to $\sigma_z$ on $Op\big( F^z_1(I^k) \big)$.
\item $\hat{\sigma}^{\perp}_z$ is $C^0$-small on all of $I^m$.
\item $F^z_t$ is $C^0$-small.
\item $F^z_t=id_{I^m}$ and $\hat{\sigma}_z=0$ on $Op(\partial I^m)$.
\item $F^z_t=id_{I^m}$ and $\hat{\sigma}_z=0$ on all of $I^m$ for $z \in Op(\partial I^q)$.
\end{itemize}
\end{theorem}

The rest of the paper is devoted to the proofs of the local relative Theorems \ref{localized holonomic approximation lemma for $l$-holonomic sections}, \ref{localized holonomic approximation lemma for perp-holonomic sections}, \ref{parametric localized holonomic approximation lemma for $l$-holonomic sections} and \ref{parametric localized holonomic approximation lemma for perp-holonomic sections}.

\section{Geometry of jet spaces}\label{Geometry of jet spaces}

\subsection{Jets as Taylor polynomials}\label{Taylor polynomials}
The reduction carried out in Section \ref{Localization of the problem} leads us to study the local space $J^r(\bR^m, \bR^n)$. We begin by giving an explicit description of this space in terms of Taylor polynomials. This description is useful both for intuition and for the explicit computations to be carried out later on.

Given a point $x \in \bR^m$ and given $n$ real polynomials $p_1(X), \ldots , p_n(X) \in \bR[X]$ in $m$ variables $X=(X_1, \ldots , X_m)$ of degree $\leq r$, set $s(x) \in J^r(\bR^m, \bR^n)$ to be the $r$-jet at $x$ of the germ 
\[ y \mapsto \big(p_1(y-x), \ldots , p_n(y-x) \big) \in \bR^n , \quad y \in Op(x) \subset \bR^m.\]

 This assignment yields a trivialization $J^r(\bR^m, \bR^n) \simeq \bR^m  \times (\cP_r)^n$, $s(x) \leftrightarrow \big(x,p_1(X), \ldots , p_n(X)\big)$, where $\cP_r =\{ p(X) \in \bR[X] : \, \, \text{deg}\big( p(X) \big) \leq r \}$. Indeed, if $h:Op(x) \subset \bR^m \to \bR^n$ is the germ of a smooth function at the point $x \in \bR^m$ and $\big(p_1(X), \ldots, p_n(X) \big)$ is its linear Taylor approximation of order $r$ centered at $x$, then the above construction yields $s(x)=j^r(h)(x)$. In this way we think of an arbitrary section $s:I^m \to J^r(\bR^m, \bR^n)$ as a familiy $s(x)$ of Taylor polynomials of degree $\leq r$ parametrized by the point $x \in I^m$.  
 
Observe that for $l<r$ we obtain an induced trivialization $(p^r_l)^{-1}(0)  \simeq \bR^m \times (\cP_{l,r})^n$, where we recall the affine bundle $p^r_l:J^r(\bR^m,\bR^n) \to J^l(\bR^m,\bR^n)$, $\sigma \mapsto \sigma^{(l)},$ and we denote by $\cP_{l,r}\subset \cP_r $ the space of real polynomials in $m$ variables $X=(X_1, \ldots , X_m)$ which are sums of monomials of degree strictly greater than $l$ and at most $r$. If we further denote by $\cH_j=\cP_{j-1,j}$ the space of homogeneous polynomials in $m$ variables $X=(X_1, \ldots , X_m)$ of degree exactly $j$, then from the degree splitting $\cP_r=\cH_0 \times \cH_1 \times \cdots \times \cH_r$ we get an induced decomposition $J^r(\bR^m, \bR^n) \simeq \bR^m \times (\cH_0)^n \times (\cH_1)^n \times \cdots \times (\cH_r)^n$ into homogeneous components. By the homogeneous component of order $j$ of an $r$-jet $s(x) \in J^r(\bR^m, \bR^n)$ at $x \in \bR^m$ we will mean the $(\cH_j)^n$ entry corresponding to $s(x)$ under the above decomposition. The projection $p^r_l$ simply forgets the homogeneous components of degree $>l$ and so we have a similar decomposition $(p^r_l)^{-1}(0) \simeq \bR^m \times (\cH_{l+1})^n \times \cdots (\cH_r)^n$. 

We next consider the trivialization $J^r(\bR^m, \bR^n) \simeq \bR^m \times (\cP_r)^n$ in the context of primitive sections, as defined in Section \ref{Holonomic trivialization}. Let $\tau \subset T\bR^m$ be a hyperplane field. We can specify a co-orientation of $\tau$ by choosing a family of unit vectors $u_x \in \bR^m$ which are orthogonal to $\tau_x$ with respect to the usual Euclidean inner product $\langle \cdot \, , \cdot \rangle$. Set $l_x: \bR^m \to \bR$ to be the linear function $l_x(\cdot)= \langle \cdot \, , u_x \rangle$. Given a point $x \in \bR^m$ and a vector $v=(v_1, \ldots , v_n) \in  \bR^n$, define $s(x) \in J^r(\bR^m, \bR^n)$ to be the $r$-jet at the point $x \in \bR^m$ of the germ 
\[ y \mapsto \big( l_x(y-x)\big)^r \cdot (v_1, \ldots, v_n) \in \bR^n  , \quad  y \in Op(x) \subset \bR^m.\]

In other words, we repeat the above construction $s(x) \leftrightarrow \big(x, p_1(X), \ldots , p_n(X) \big)$ with the polynomials $p_j(X)=v_j \cdot \big(l_x(X)\big)^r$, which are all multiples of the $r$-th power of a linear function with kernel $\tau_x$. Observe that the resulting $r$-jet satisfies $s(x)^{\perp}=0$ with respect to the hyperplane field $\tau_x$. In fact, the choice of co-orientation $u_x$ determines a trivialization of the space of sections which are primitive with respect to $\tau$, namely $(p^r_{\perp})^{-1}(0) \simeq \bR^m \times \bR^n$, $s(x) \leftrightarrow (x,v)$, where we recall the affine fibration $p^r_\perp : J^r(\bR^m, \bR^n) \to J^{\perp}(\bR^m, \bR^n)$, $\sigma \mapsto \sigma^{\perp}$. 

\subsection{Linear structure}
For a general fibre bundle $p:X \to M$ we have affine bundle structures $p^r_l:X^{(r)} \to X^{(l)}$ for each $l<r$, but there is no invariantly defined linear structure on the bundle $p^r_l$. Equivalently, in general we cannot invariantly define inclusions $X^{(l)} \subset X^{(r)}$. The reason is that the chain rule for derivatives of order $r$ involves all derivatives of order $ \leq r$ and therefore a change of coordinates will mix up the jet components of different orders. 

Nevertheless, in the case $X^{(r)}=J^r(\bR^m, \bR^n)$ we have a canonical linear structure arising from the linear structure on $\bR^n$. Explicitly, if $s(x) , \sigma(x) \in J^r(\bR^m, \bR^n)$ are $r$-jets at the point $x \in \bR^m$ corresponding to germs $h,g:Op(x) \subset \bR^m \to \bR^n$ and if $a,b \in \bR$ are any two real numbers, then we can define $a \cdot s(x)+b \cdot \sigma(x) \in J^r(\bR^m, \bR^n)$ to be the $r$-jet at the point $x \in \bR^m$ corresponding to the germ $ah+bg:Op(x) \subset \bR^m \to \bR^n$. We can therefore equip the $r$-jet bundle $p^r: J^r(\bR^m, \bR^n) \to \bR^m$ with the structure of a vector bundle. Similarly, for all $l<r$ we can endow each of the projections $p^r_l:J^r(\bR^m, \bR^n) \to J^l(\bR^m, \bR^n)$ with vector bundle structures. However, we will reserve the addition sign to denote the linear structure on the bundle $p^r$. In terms of the trivialization $J^r(\bR^m, \bR^n) \simeq \bR^m \times (\cP_r)^n$, this linear structure corresponds to addition of polynomials in $\cP_r$. 

\subsection{Reduction to the $l=r-1$ case}\label{Reduction to the case $l=r-1$}
In order to deduce the localized holonomic approximation lemma for $l$-holonomic sections  \ref{localized holonomic approximation lemma for $l$-holonomic sections} as a corollary of the localized holonomic approximation lemma for $\perp$-holonomic sections \ref{localized holonomic approximation lemma for perp-holonomic sections}, it is useful to first reduce to the case $l=r-1$. This reduction is accomplished by the following inductive argument.

\begin{lemma}\label{reduction to $l=r-1$} Suppose that there exists $j\leq r$ such that Theorem \ref{localized holonomic approximation lemma for $l$-holonomic sections} holds for all $r$ and $l$ such that $r-l<j$. Then it also holds for all $r$ and $l$ such that $r-l<j+1$.
\end{lemma}
\begin{remark} Before we dive into the proof, we recall the notion of pullbacks and pushforwards in jet spaces. Suppose that $F:\bR^m \to \bR^m$ is a diffeomorphism, $F(x)=y$ and $h:Op(y) \to \bR^n$ is a germ of a smooth function at the point $y \in \bR^m$. Then $h \circ F : Op(x) \to \bR^n$ is a germ of a smooth function at $x \in \bR^m$. This assignment defines a pullback map $F^* : J^r(\bR^m, \bR^n) \to J^r(\bR^m, \bR^n)$ which covers $F^{-1}$. Similarly, we define the pushforward $F_*=(F^{-1})^* : J^r(\bR^m, \bR^n) \to J^r(\bR^m, \bR^n)$ which covers $F$.
\end{remark}
\begin{proof}[Proof of Lemma \ref{reduction to $l=r-1$}]
Let $\sigma:I^m \to J^r(\bR^m, \bR^n)$ be a section such that $\sigma=0$ on $Op(\partial I^m)$ and such that $\sigma^{(r-j)}=0$ on all of $I^m$. Let $\mu=\sigma^{(r-1)}:I^m \to J^{r-1}(\bR^m, \bR^n)$ be its $(r-1)$-jet component. Then we also have $\mu=0$ on $Op(\partial I^m)$ and $\mu^{(r-j)}=0$ on all of $I^m$. Observe that $r-j=(r-1)-(j-1)$ and therefore by assumption there exists a $C^0$-small isotopy $H_t:I^m \to I^m$ such that $H_t=id_{I^m}$ on $Op(\partial I^m)$ and a holonomic section $\hat{\mu}: I^m \to J^{r-1}(\bR^m, \bR^n)$ such that $\hat{\mu}$ is $C^0$-close to $\mu$ on $Op\big(H_1(I^k)\big)$, such that $\hat{\mu}^{(r-j)}$ is $C^0$-small and such that $\mu=0$ on $Op(\partial I^m)$. 

Since $\hat{\mu}$ is holonomic, we have $\hat{\mu}=j^{r-1}(h)$ for some function $h:I^m \to \bR^m$. There exists a unique section $\nu:I^m \to J^r(\bR^m, \bR^n)$ such that $\nu^{(r-1)}=0$ and such the homogeneous order $r$ component of $\nu$ is equal to the homogeneous order $r$ component of the section $\sigma-j^r(h)$. Observe that the pullback $(H_1)^*\nu$ by the diffeomorphism $H_1$ also has zero $(r-1)$-jet component. We can therefore apply once again our inductive hypothesis to ensure the existence of a $C^0$-small isotopy $\widetilde{H}_t: I^m \to I^m$ such that $\widetilde{H}_t=id_{I^m}$ on $Op(\partial I^m)$ and a holonomic section $\hat{\nu}:I^m \to J^r(\bR^m, \bR^n)$ such that $\hat{\nu}$ is $C^0$-close to $(H_1)^*\nu$ on $Op\big(\widetilde{H}_1(I^k) \big)$, such that $\hat{\nu}^{(r-1)}$ is $C^0$-small and such that $\hat{\nu}=0$ on $Op(\partial I^m)$.  

Set $F_t=H_t \circ \widetilde{H}_t$ and $\hat{\sigma}=j^{r}(h)+(H_1)_*(\hat{\nu})$. Then we can rephrase our above conclusions by stating that $F_t$ is $C^0$-small isotopy such that $F_t=id_{I^m}$ on $Op(\partial I^m)$ and that $\hat{\sigma}$ is a holonomic section of $J^r(\bR^m, \bR^n)$ such that $\hat{\sigma}$ is $C^0$-close to $\sigma $ on $Op\big( F_1(I^m) \big)$, such that $\hat{\sigma}^{(r-j)}$ is $C^0$-small and such that $\hat{\sigma}=0$ on $Op(\partial I^m)$. This is exactly what we wanted. \end{proof}

\begin{figure}[h]
\includegraphics[scale=0.62]{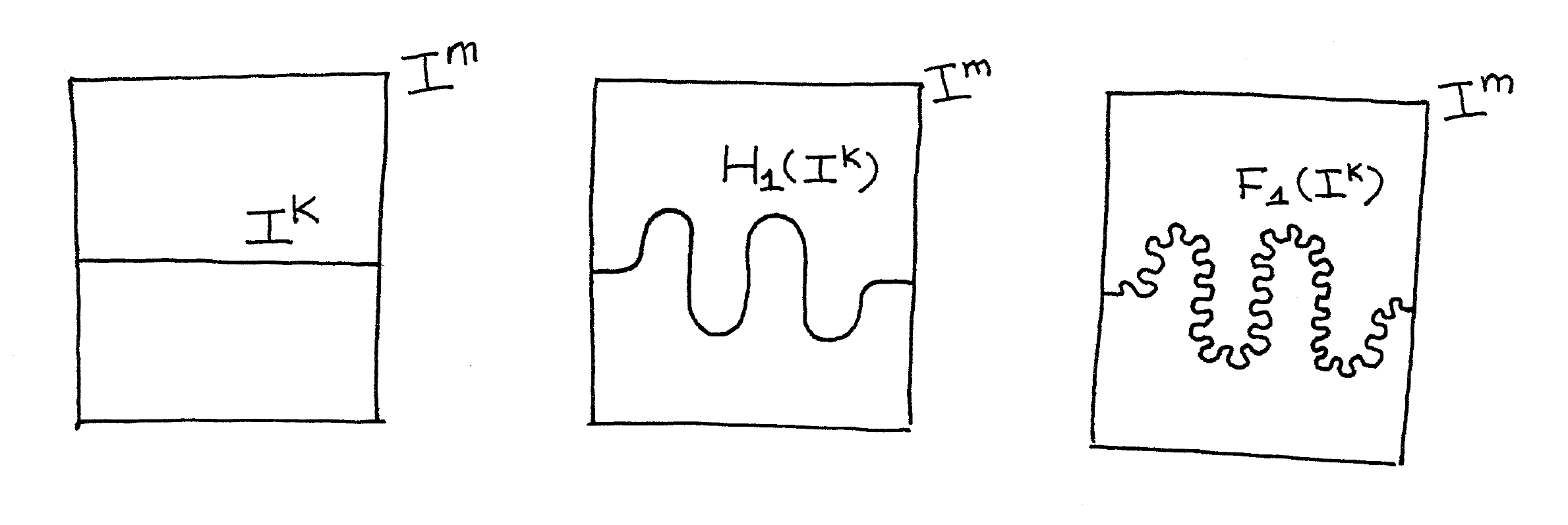}
\caption{Wiggling the wiggles with $F_t=H_t \circ \widetilde{H}_t$.}
\label{Wiggling the wiggles}
\end{figure}

Lemma \ref{reduction to $l=r-1$} also holds in families, with the same proof. The only difference is that one needs to add a parameter everywhere in the notation. Therefore, it also suffices to prove the parametric local relative Theorem \ref{parametric localized holonomic approximation lemma for $l$-holonomic sections} in the case $l=r-1$.

\subsection{Decomposition into primitive sections}\label{Decomposition into primitive sections}
To reduce the $l=r-1$ case of the localized holonomic approximation lemma for $l$-holonomic sections to the localized holonomic approximation lemma for $\perp$-holonomic sections we need to consider decompositions of $r$-jet sections with zero $(r-1)$-jet component into sums of primitive sections. The following discussion closely resembles the theory of principal decompositions in jet spaces invented by Gromov in \cite{G86} in the context of convex integration and further fleshed out by Spring in \cite{S98}.  

Given a fixed holonomic section $\hat{s}:M \to X^{(r)}$ of the $r$-jet bundle of a fibre bundle $p:X \to M$, recall that primitive sections are the local analogues of $\perp$-holonomic sections whose $\perp$-jet equals $\hat{s}^{\perp}$. We repeat the precise definition for convenience.
\begin{definition} A section $\sigma: I^m \to J^r(\bR^m, \bR^n)$ is called primitive with respect to a hyperplane field $\tau \subset TI^m$ if $\sigma^{\perp}=0$.
\end{definition}
We are particularly interested in sections which are primitive with respect to the hyperplane fields $\tau_\beta= \ker( dx_{\beta_1} + \cdots + dx_{\beta_k})$, where $\beta$ ranges over all multi-indices such that $1 \leq \beta_j \leq m$ and such that $|\beta|=k \leq r$, up to permutation. There are of course redundancies among the $\tau_\beta$, but this is not important. We remark that the hyperplane fields $\tau_\beta$ are constant and hence integrable. Given a section $\sigma : I^m \to J^r(\bR^m, \bR^n)$ such that $\sigma^{(r-1)}=0$, our goal is to obtain a decomposition $\sigma= \sum_\beta \sigma_\beta$ where each section $\sigma_\beta$ is primitive with respect to $\tau_\beta$, see Figure \ref{decomposition}. Moreover, we want this decomposition to be well-behaved in a sense that is made precise below.  

For this purpose we invoke the following simple polynomial identity, which the author found in \cite{B14} but which may well be classical. Consider the formula
\[ X_1 X_2 \cdots X_r = \frac{(-1)^r}{r!}\sum_U (-1)^{|U|}\Big( \sum_{u \in U} X_u\Big)^r, \]
where the sum ranges over all subsets $U \subset \{ 1, 2, \ldots , r\}$. Recall from Section \ref{Taylor polynomials} that we can think of an $r$-jet $\sigma(x) \in J^r(\bR^m, \bR^n)$ at the point $x \in I^m$ such that $\sigma^{(r-1)}(x)=0$ as a homogeneous Taylor polynomial $\big(p_1(X), \ldots , p_n(X) \big)$ of degree $r$ centered at $x$. Each such polynomial $p_j(X)$ can be written uniquely as a sum of monomials: $p_j(X) = \sum_\alpha a_\alpha X_{\alpha_1} \cdots X_{\alpha_r}$, where $\alpha$ ranges through all multi-indices $\alpha=(\alpha_1, \ldots , \alpha_r)$ such that $1 \leq \alpha_j \leq m$, up to permutation. Hence we can write
\[ p_j(X) = \sum_\alpha a_\alpha X_{\alpha_1}\cdots X_{\alpha_r} = \sum_\alpha \frac{(-1)^ra_\alpha}{r!}  \sum_U (-1)^{|U|} \Big( \sum_{u \in U} X_{\alpha_u} \Big)^r\] \[  = \sum_\beta \Big(  \sum_{ (\alpha_u)=\beta } \frac{(-1)^{r+ |\beta|}a_\alpha}{r!} \Big) \big( X_{\beta_1} + \cdots + X_{\beta_k} \big)^r,\] 
where the inner sum ranges over all pairs $(\alpha, U)$ such that $(\alpha_u)_{u \in U}=\beta$. Observe that the homogeneous degree $r$ monomial $(X_{\beta_1} + \cdots + X_{\beta_k} )^r$ corresponds to an $r$-jet which is primitive with respect to $\tau_\beta=\ker(dx_{\beta_1} + \cdots + dx_{\beta_k})$. We have therefore proved the following.
\begin{lemma}\label{decomposition} Given a section $\sigma:I^m \to J^r(\bR^m, \bR^n)$ such that $\sigma^{(r-1)}=0$, we can write $\sigma= \sum_\beta \sigma_\beta$ for sections $\sigma_\beta : I^m \to J^r(\bR^m, \bR^n)$ such that the following properties hold.
\begin{itemize}
\item Each section $\sigma_\beta$ is primitive with respect to $\tau_\beta$.
\item Each section $\sigma_\beta$ depends smoothly on $\sigma$.
\item If $\sigma=0$ on $Op(A)$ for some closed subset $A \subset I^m$, then $\sigma_\beta=0$ on $Op(A)$ for all $\beta$.
\end{itemize}
\end{lemma}
\begin{remark} Observe that the number of indices $\beta$ appearing in the sum only depends on $m$ and $r$.
\end{remark}
\begin{figure}[h]
\includegraphics[scale=0.4]{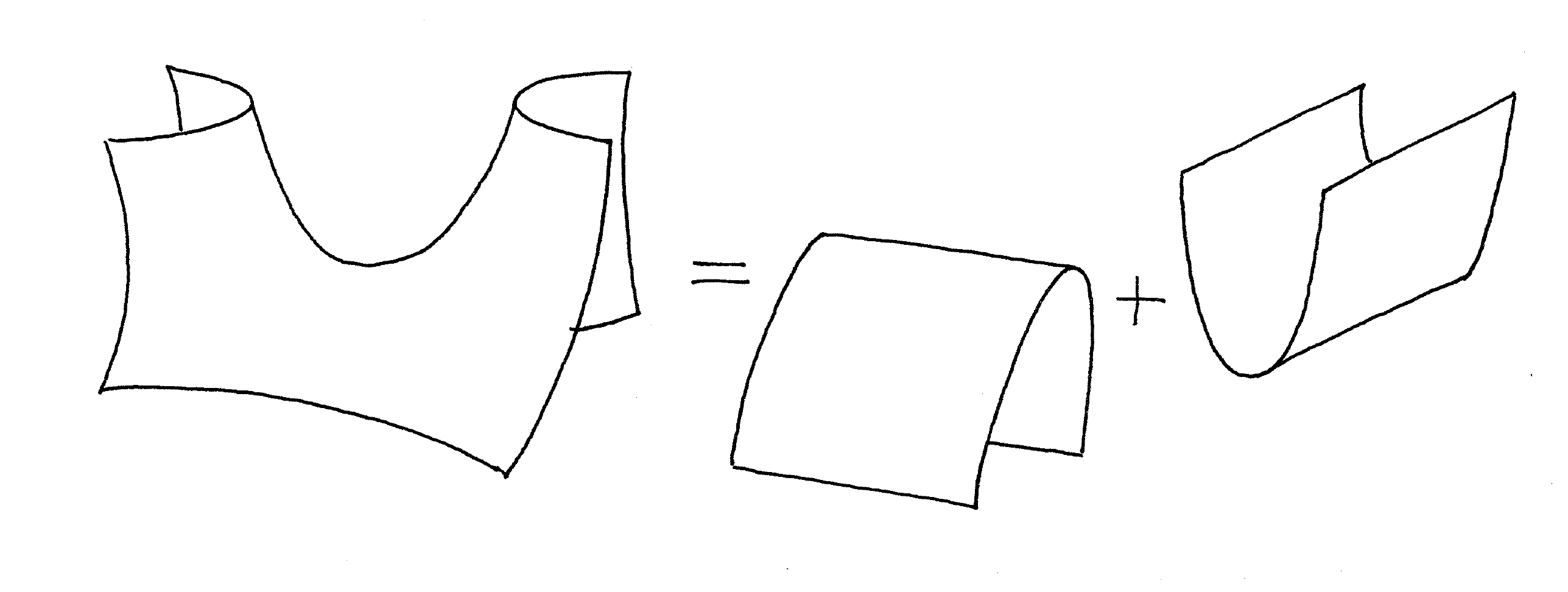}
\caption{Example: we can decompose the homogeneous degree $2$ polynomial $x^2-y^2$ into the sum of $x^2$ and $-y^2$.}
\label{decomposition}
\end{figure}

\subsection{Reduction to the primitive case}
We are now ready to reduce the localized holonomic approximation lemma for $l$-holonomic sections \ref{localized holonomic approximation lemma for $l$-holonomic sections} to the localized holonomic approximation lemma for $\perp$-holonomic sections \ref{localized holonomic approximation lemma for perp-holonomic sections}. Recall first that by the discussion of Section \ref{Reduction to the case $l=r-1$}, we only need to prove Theorem \ref{localized holonomic approximation lemma for $l$-holonomic sections} in the case $l=r-1$. Let us therefore assume that Theorem \ref{localized holonomic approximation lemma for perp-holonomic sections} holds and let $\sigma:I^m \to J^r(\bR^m, \bR^n)$ be a section such that $\sigma=0$ on $Op(\partial I^m)$ and such that $\sigma^{(r-1)}=0$ on all of $I^m$. Lemma \ref{decomposition} gives us a decomposition $\sigma= \sum_\beta \sigma_\beta$, where $\sigma_\beta:I^m \to J^r(\bR^m, \bR^n)$ is a section such that $\sigma_\beta=0$ on $Op(\partial I^m)$ and such that $\sigma_\beta^\perp=0$ with respect to the hyperplane field $\tau_\beta \subset TI^m$ defined in Section \ref{Decomposition into primitive sections}. We will inductively construct holonomic approximations for the partial sums of the decomposition $\sigma= \sum_\beta\sigma_\beta$. The main point in the following argument is that if an $r$-jet section has a $C^0$-small $\perp$-jet component, then in particular it also has a $C^0$-small $(r-1)$-jet component. 

Let $\beta^1, \beta^2, \ldots, \beta^N$ be an ordering of the multi-indices $\beta$ appearing in the decomposition $\sigma= \sum_\beta \sigma_\beta$ and denote by $\sigma_1, \ldots , \sigma_N$ and $\tau_1, \ldots , \tau_N$ the corresponding sections $\sigma_\beta$ and hyperplane fields $\tau_\beta$. We begin by applying Theorem \ref{localized holonomic approximation lemma for perp-holonomic sections} to the section $\sigma_1$. We obtain a $C^0$-small isotopy $F^1_t:I^m \to I^m$ such that $F^1_t=id_{I^m}$ on $Op(\partial I^m)$ and a holonomic section $\hat{\sigma}_1:I^m \to J^r(\bR^m, \bR^n)$ such that $\hat{\sigma}_1$ is $C^0$-close to $\sigma_1$ on $Op\big( F_1(I^k) \big)$, such that $\hat{\sigma}_1^{(r-1)}$ is $C^0$-small and such that $\hat{\sigma}_1=0$ on $Op(\partial I^m)$. This concludes the base case of the induction. 

Suppose that for some $j<N$ we have constructed a $C^0$-small isotopy $F^j_t:I^m \to I^m$ such that $F^j_t=id_{I^m}$ on $Op(\partial I^m)$ and a holonomic section $\hat{\sigma}_j:I^m \to J^r(\bR^m, \bR^n)$ such that $\hat{\sigma}_j$ is $C^0$-close to $\sum_{i \leq j} \sigma_i$ on $Op \big( F^j_1(I^k) \big)$, such that $\hat{\sigma}_j^{(r-1)}$ is $C^0$-small and such that $\hat{\sigma}_j=0$ on $Op(\partial I^m)$. Apply Theorem \ref{localized holonomic approximation lemma for perp-holonomic sections} to the section $(F^j_1)^*\sigma_{j+1}:I^m \to J^r(\bR^m, \bR^n)$, which satisfies $(F^j_1)^*\sigma_{j+1}=0$ on $Op(\partial I^m)$ and $(F^j_1)^*\sigma_{j+1}^{\perp}=0$ on all of $I^m$ with respect to the hyperplane field $(F^j_1)^* \tau_{j+1}$. We obtain a $C^0$-small isotopy $\widetilde{F}_t:I^m \to I^m$ such that $\widetilde{F}_t=id_{I^m}$ on $Op(\partial I^m)$ and a holonomic section $\widetilde{\sigma}:I^m \to J^r(\bR^m, \bR^n)$ such that $\widetilde{\sigma}$ is $C^0$-close to $(F^j_1)^*\sigma_{j+1}$ on $Op\big(\widetilde{F}_1(K) \big)$, such that $\widetilde{\sigma}^{(r-1)}$ is $C^0$-small and such that $\widetilde{\sigma}=0$ on $Op(\partial I^m)$. Set $F^{j+1}_t=F^j_t \circ \widetilde{F}_t$ and $\hat{\sigma}_{j+1}= \hat{\sigma}_j + (F^j_1)_*\widetilde{\sigma}$. This completes the inductive step. 

At the last step we obtain a $C^0$-small isotopy $F_t=F^N_t:I^m \to I^m$ such that $F_t=id_{I^m}$ on $Op(\partial I^m)$ and a holonomic section $\hat{\sigma}=\hat{\sigma}_N$ such that $\hat{\sigma}$ is $C^0$-close to $\sigma= \sum_{i=1}^N \sigma_i$ on $Op\big( F_1(I^k) \big)$, such that $\hat{\sigma}^{(r-1)}$ is $C^0$-small and such that $\hat{\sigma}=0$ on $Op(\partial I^m)$. This is exactly what we wanted. We have thus sucessfully reduced Theorem \ref{localized holonomic approximation lemma for $l$-holonomic sections} to Theorem \ref{localized holonomic approximation lemma for perp-holonomic sections}. 

It remains to discuss the reduction of the parametric localized holonomic approximation lemma for $l$-holonomic sections  \ref{parametric localized holonomic approximation lemma for $l$-holonomic sections} to the parametric localized holonomic approximation lemma for $\perp$-holonomic sections  \ref{parametric localized holonomic approximation lemma for perp-holonomic sections}. However, the proof only differs in notation, namely one just needs to add a parameter everywhere. The key point here is that given a family $\sigma_z:I^m \to J^r(\bR^m, \bR^n)$ of sections such that $\sigma_z^{(r-1)}=0$, the decomposition given by Lemma \ref{decomposition} depends smoothly on the parameter $z$.

\section{Transversality adjustment}\label{transversality adjustment}

\subsection{The transversality condition}

By the reductions carried out in Sections \ref{Localization of the problem} and \ref{Geometry of jet spaces}, we are left with the task of proving Theorems  \ref{localized holonomic approximation lemma for perp-holonomic sections} and  \ref{parametric localized holonomic approximation lemma for perp-holonomic sections}, the local relative holonomic approximation lemmas for $\perp$-holonomic sections. The strategy of proof, as in classical holonomic approximation, is to take advantage of the room provided by the positive codimension of $I^k$ in $I^m$, where $k<m$. This room is used to interpolate between the Taylor polynomials determined by the non-holonomic section that we wish to approximate. More precisely, the idea is to wiggle the subset $I^k \subset I^m$ back and forth in the ambient space $I^m$ and interpolate between Taylor polynomials along the wiggles. However, in order to obtain the fine estimates needed for the desired control on the $\perp$-jet component, our wiggles must be parallel to the hyperplane field $\tau$ under consideration. We therefore run into difficulties when $\tau$ is tangent to the subset $I^k$ which we want to wiggle. In this section we will perform yet another reduction, so that we only need to consider hyperplane fields $\tau$ which are transverse to $I^k$. 

The idea is to further localize the problem by subdividing the cube $I^m$ into very small subcubes, on each of which the hyperplane field $\tau$ is almost constant. We show in Section \ref{almost tangent} below that on the subcubes where the hyperplane field $\tau$ is almost tangent to $I^k$, the desired holonomic approximation can be explicitly constructed by hand. Moreover, in this case no wiggling is necessary. The accuracy of the approximation will depend on the extent to which $\tau$ is almost tangent to $I^k$, but given a fixed degree of accuracy desired we can always restrict our attention to those subcubes on which the angle between $\tau$ and $I^k$ is sufficiently small. We explain precisely how to achieve this transversality adjustment in Section \ref{reduction to transverse case}. On the remaining cubes, the hyperplane field $\tau$ is transverse to $I^k$ and therefore we can perform the wiggling parallel to $\tau$ described in the previous paragraph. This last step is carried out in Section \ref{holonomic approximation with controlled cutoff}. We illustrate our strategy in Figure \ref{two-step}.

\begin{figure}[h]
\includegraphics[scale=0.7]{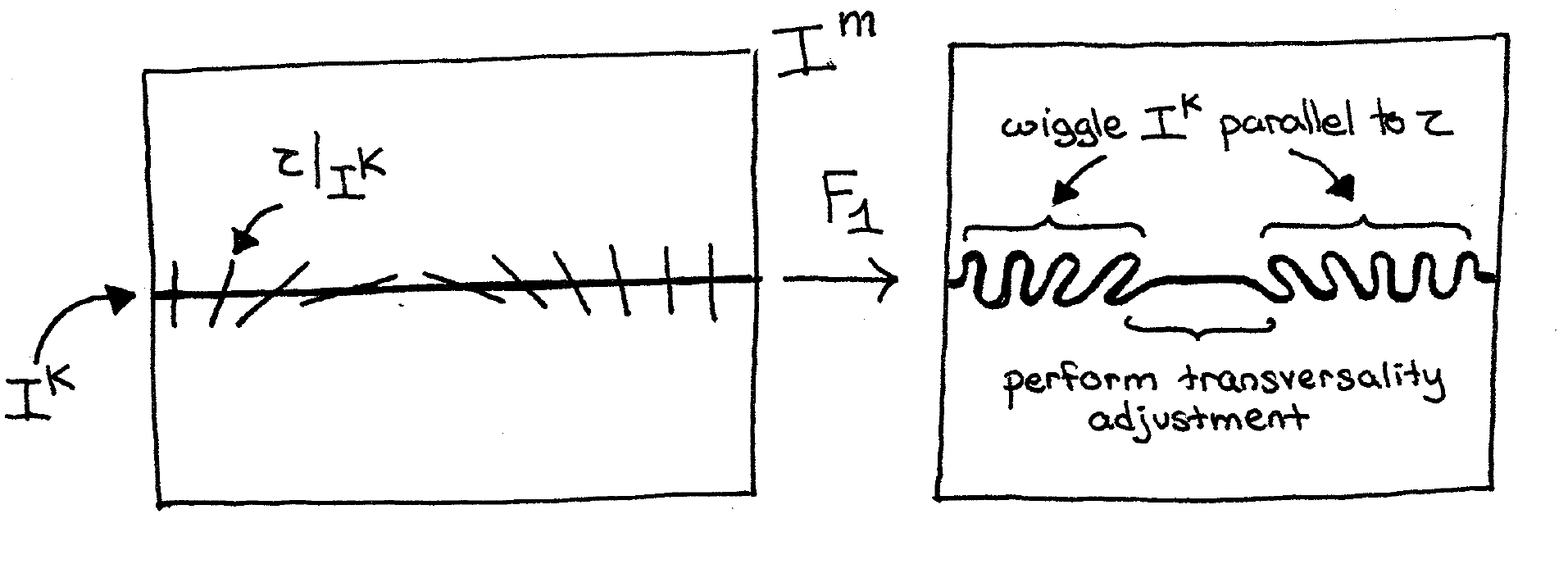}
\caption{The two steps: a transversality adjustment and a wiggling parallel to $\tau \subset TI^m$.}
\label{two-step}
\end{figure}

\subsection{Almost tangent hyperplane fields}\label{almost tangent}
Let $V,W \subset \bR^m$ be two linear subspaces of the same dimension. Recall that the angle $\measuredangle(V, W )$ between $V$ and $W$ is defined as $\measuredangle(V,W)=\sup_{v \in W \setminus 0} \big( \inf_{w \in V \setminus 0 } \measuredangle(v,w) \big)$. When $\dim V < \dim W$, we set $\measuredangle(V,W)= \inf_U \measuredangle(V,U)$, where the infimum is taken over all linear subspaces $U \subset W$ such that $\dim U = \dim V$. Equivalently, we have $\measuredangle(V,W)= \inf_U \measuredangle(U,W)$, where the infimum is taken over all subspaces $U \supset V$ such that $\dim U = \dim W$. For any two distributions $\tau, \eta \subset TI^m$ we define $\measuredangle(\tau, \eta)= \sup_{x \in I^m} \measuredangle( \tau_x, \eta_x)$. The main goal of this section is to establish the following local calculation, where we think of the hyperplane $H=\bR^{m-1} \times 0 \subset \bR^m$ as a constant hyperplane field on $I^m$.
\begin{lemma}\label{calculation for almost tangent}
Fix $k<m$. Let $\sigma:I^m \to J^r(\bR^m,\bR^n)$ be a section and let $\tau \subset TI^m$ be a hyperplane field such that the following properties hold.
\begin{itemize}
\item $\sigma^\perp=0$ with respect to $\tau$.
\item $\sigma=0$ on $Op(\partial I^m)$.
\end{itemize}
Then for every $\delta>0$ there exists a holonomic section $\hat{\sigma}:I^m \to J^r(\bR^m, \bR^n)$ such that the following properties hold.
\begin{itemize}
\item $\text{dist}_{C^0}(\hat{\sigma},\sigma) \leq C\, || \sigma ||_{C^r}  \big(  \measuredangle(\tau,H) + \delta \big)$ on $Op(I^{k})$.
\item $||\hat{\sigma}^{\perp}||_{C^0} \leq C\, || \sigma ||_{C^r} \big( \measuredangle(\tau,H) + \delta \big) $ on all of $I^m$.
\item $\hat{\sigma}=0$ on $Op(\partial I^m)$.
\end{itemize}
\end{lemma}
\begin{remark}\label{remarks on calculation} $ $
 \begin{enumerate} 
 \item The constant $C>0$ only depends on $m$ and $r$. We can extract an explicit upper bound for $C$ from the proof if we so desire, but this is not important.
 \item To be more precise, for a function germ $h:Op(x) \subset \bR^m \to \bR$ we define $||j^r(h)(x)||=\text{sup} ||\partial_\alpha h(x) ||$, where the supremum is taken over all multi-indices $\alpha$ of order $|\alpha| \leq r$. Similarly, we define $||j^\perp(h)(x)||=\text{sup}  || \partial_\nu( \partial_\alpha h )(x)||$, where the supremum is taken over all multi-indices $\alpha$ of order $|\alpha|<r$ and over all unit vectors $\nu \in \tau_x$.  For a section $s:I^m \to J^r(\bR^m, \bR^n)$ we set $||s||_{C^0}=\sup_{x\in I^m} || s(x)||$ and $|| s^{\perp} ||_{C^0}= \sup_{x \in I^m} ||s^{\perp}(x)||$. These are the $C^0$ norms which appear in the statement of the Lemma. There are of course many other equivalent definitions, but they all differ by a constant which can be absorbed into $C$.
 \item We can define the $C^r$ norm in a similar way. Think of $\sigma$ as a family of germs $y \mapsto h_x(y)$, $y \in Op(x)$, parametrized by $x \in I^m$. We can differentiate the vector $h_x(y) \in \bR^n$ with respect to $x$ or with respect to $y$. Set $||\sigma||_{C^r}=\sup||  \frac{\partial}{\partial x_\beta}\frac{\partial}{\partial y_{\alpha}}h|_{(x,x)} ||$, where the supremum is taken over all multi-indices $\alpha,\beta$ of orders $|\alpha|,|\beta|\leq r$ and all points $x \in I^m$.
 \end{enumerate}
\end{remark}
\begin{proof}
Throughout the proof $C>0$ will denote a constant, depending only on $m$ and $r$, but which might be replaced with a bigger such constant whenever necessary. Assume without loss of generality that the angle between $\tau$ and $H$ is small, say $\measuredangle(\tau, H) < \pi/4$. Let $u_x \in \bR^m$, $x \in I^m$, be the unique field of unit vectors such that for every $x \in I^m$ we have $u_x \perp \tau_x$ and $\measuredangle(\tau_x,H_x)=\measuredangle(u_x, \partial_m)$, where $\partial_m=(0, \ldots, 0,1) \in \bR^m$.  We know from Section \ref{Taylor polynomials} that for every $x \in I^m$, $\sigma(x) \in J^r(\bR^m, \bR^n)$ is the $r$-jet at $x$ of a germ $y \mapsto \big( l_x(y-x) \big)^r \cdot \big(  v_1(x), \ldots , v_n(x) \big) \in \bR^n$, $y \in Op(x) \subset \bR^m$, where $l_x:\bR^m \to \bR$ is the linear function $l_x(\cdot)=\langle \cdot \, ,u_x \rangle$ and $v=(v_1, \ldots, v_n):I^m \to \bR^n$ is a function such that $v=0$ on $Op(\partial I^m)$. Fix once and for all a cutoff function $\psi: \bR \to \bR$ such that $\psi(t)=1$ when $|t|<1/2$ and $\psi(t)=0$ when $|t|>1$. For $\delta>0$ small, set $\psi_\delta(t)=\psi(t/\delta)$. Define a holonomic section $\hat{\sigma}:I^m \to J^r(\bR^m, \bR^n)$ by $\hat{\sigma}=j^r(h)$, where $h:I^m \to \bR^n$ is the function
\[ h(x)=\psi_{\delta}(x_m) \cdot x_m^r  \cdot \big(v_1(x) , \ldots , v_n(x) \big), \qquad x=(x_1, \ldots , x_m) \in I^m. \]

To verify that $\hat{\sigma}$ satisfies the desired properties, we introduce an auxiliary section $s:I^m \to J^r(\bR^m, \bR^n)$ whose $r$-jet $s(x) \in J^r(\bR^m, \bR^n)$ at the point $x =(x_1, \ldots , x_m)\in I^m$ corresponds to the germ $y \mapsto (y_m-x_m)^r \cdot \big(v_1(x), \ldots ,v_n(x) \big) \in \bR^n$, $y=(y_1, \ldots , y_m) \in Op(x) \subset \bR^m$. Indeed, $\measuredangle(u,\partial_m)=\measuredangle(\tau, H)$, and hence $\text{dist}_{C^0}(s,\sigma) \leq C \, ||\sigma   ||_{C^r} \, \, \measuredangle(\tau,H)$ on all of $I^m$. On the other hand, $\text{dist}(\hat{\sigma},s) \leq C \, || \sigma ||_{C^r} \, \delta $ on $\{ |x_m | < \delta/2 \} \subset \bR^m$ and we are free to choose $\delta$ as small as desired. This proves the first property stated in Lemma \ref{calculation for almost tangent}. The third property holds by inspection. It remains to prove the second property. 

We compute explicitly the partial derivatives $\partial_\alpha h$ for a multi-index $\alpha$ of order $|\alpha| \leq r$ at a point $x \in I^m$. Write $\alpha=(\beta, \gamma)$, where $\beta$ consists of $N$ indices $1\leq \beta_j<m$ and $\gamma$ consists of $M$ indices $\gamma_j=m$. Then we have the following formula.
\[ \partial_\alpha h(x) =  \sum_{j=0}^M {M \choose j} \Big[ \big( \sum_{i=0}^j {j \choose i} \psi^{(j-i)}_\delta(x_m) \cdot \frac{r!}{(r-i)!} x_m^{r-i} \big) \cdot ( \partial^{M-j}_m \partial_\beta v)(x) \Big], \]
\[ \text{with } \quad |\psi^{(j-i)}_\delta(x_m)|=|\frac{1}{\delta^{j-i}} \psi^{(j-i)}\big(\frac{x_m}{\delta}\big)| \leq \frac{1}{\delta^{j-i}} || \psi ||_{C^r}. \]

For $M<r$ we can therefore bound $|| \partial_\alpha h ||_{C^0} \leq C\,  || \sigma ||_{C^r}\,  \delta$ on all of $I^m$. In particular this bound holds for all multi-indices $\alpha$ of order $|\alpha| <r$. For the multi-index $\alpha=(m,\ldots,m)$ of order $|\alpha|=r$ corresponding to the pure $r$-th derivative $\partial^r_m$ we have $|| \partial_\alpha h ||_{C^0} \leq C \, || \sigma ||_{C^r}$. Hence the inequality $|| \hat{\sigma} ||_{C^0} \leq C || \sigma ||_{C^r}$ also holds.

Fix an index $\alpha$ of order $|\alpha|<r$ and let $\nu \in \tau_x$ be a unit vector. Write $\nu=\sum_j a_j \partial_j$ in terms of the standard basis $\partial_1, \ldots , \partial_m$ of $\bR^m$. Observe that $|a_m|=|\langle \nu, \partial_m \rangle| =|\langle \nu, \partial_m-u_x \rangle| \leq || \partial_m - u_x || \leq  \, \measuredangle(\tau,H)$. If follows that
 \[ || \partial_\nu (\partial_\alpha h )(x)|| \leq || a_m \partial_m ( \partial_\alpha h)(x)||  + ||\sum_{j=1}^{m-1} a_j \partial_j (\partial_\alpha h)(x) || \leq |a_m| \ || \hat{\sigma} ||_{C^0}  + \sqrt{ \sum_{j=1}^{m-1} \big( \partial_j ( \partial_\alpha h)(x) \big)^2 } \]
and therefore that $| \partial_\nu ( \partial_\alpha h)(x) | \leq C \, || \sigma ||_{C^r}  \big(\, \measuredangle(\tau,H) + \delta \big)$. But the point $x \in I^m$, the multi-index $\alpha$ of order $|\alpha|<r$ and the unit vector $\nu \in \tau_x$ were all chosen arbitrarily, and therefore we have proved the remaining inequality $||\hat{ \sigma}^{\perp} ||_{C^0} \leq C \, || \sigma ||_{C^r}  \big(\, \measuredangle(\tau,H) + \delta \big)$. See Figure \ref{picture of transversality adjustment} for an illustration of the argument. \end{proof}
\begin{figure}[h]
\includegraphics[scale=0.65]{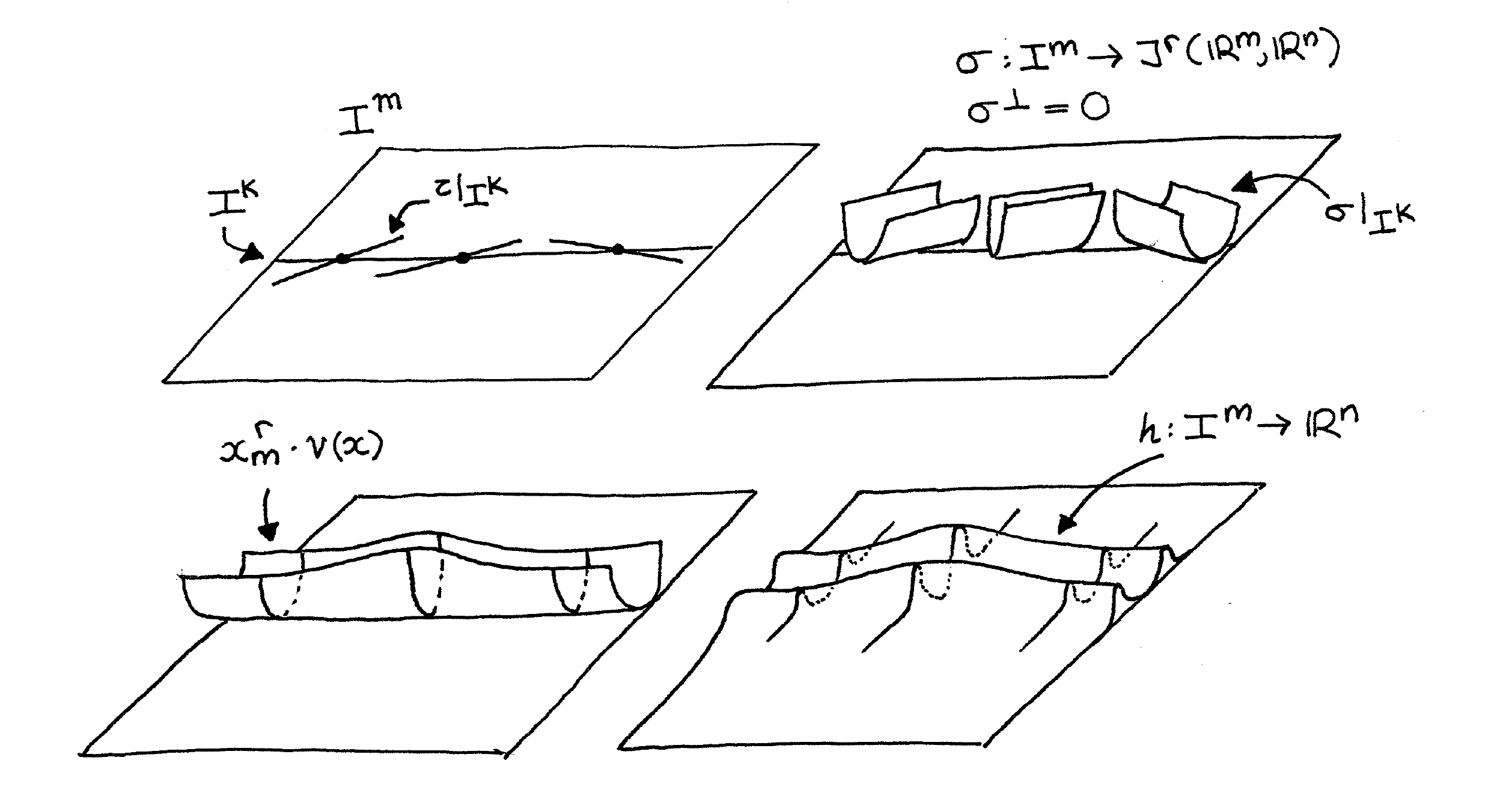}
\caption{The error in the transversality adjustment is proportional to the angle between $\tau$ and $H$. }
\label{picture of transversality adjustment}
\end{figure}
By adding a subscript everywhere in the above proof we deduce the following parametric version of Lemma \ref{calculation for almost tangent}. Observations analogous to the ones made in Remark \ref{remarks on calculation} apply.
\begin{lemma}\label{parametric calculation for almost tangent}
Fix $k<m$. Let $\sigma_z:I^m \to J^r(\bR^m,\bR^n)$ be a family of sections and let $\tau_z \subset TI^m$ be a family of hyperplane fields parametrized by the unit cube $I^q$ such that the following properties hold.
\begin{itemize}
\item $\sigma_z^\perp=0$ with respect to $\tau_z$.
\item $\sigma_z=0$ on $Op(\partial I^{k})$.
\item $\sigma_z=0$ on all of $I^m$ for $z \in Op(\partial I^q)$.
\end{itemize}
Then for every $\delta>0$ there exists a family of holonomic sections $\hat{\sigma}_z:I^m \to J^r(\bR^m, \bR^n)$ such that the following properties hold.
\begin{itemize}
\item $\text{dist}_{C^0}(\hat{\sigma}_z,\sigma_z) \leq C\, || \sigma_z ||_{C^r}  \big( \measuredangle(\tau_z,H) + \delta \big)$ on $Op(I^{k})$.
\item $||\hat{\sigma}^{\perp}_z||_{C^0} \leq C\, || \sigma_z ||_{C^r} \big(  \measuredangle(\tau_z,H)  + \delta \big)$ on all of $I^m$.
\item $\hat{\sigma}_z=0$ on $Op(\partial I^m)$.
\item $\hat{\sigma}_z=0$ on all of $I^m$ for $z \in Op(\partial I^q)$.
\end{itemize}
\end{lemma}

Of course, there is nothing special about the hyperplane $\bR^{m-1} \times 0 \subset \bR^m$, which was only fixed for concreteness. In fact, Lemmas \ref{calculation for almost tangent} and \ref{parametric calculation for almost tangent} hold in the following more general form. Observe first that we may replace the constant hyperplane field $H$ by any hyperplane field $\widetilde{H} \subset TI^m$ such that $\bR^k \times 0^{m-k} \subset \widetilde{H}_x$ at all points $x \in I^m$. To see this it suffices to consider a local change of coordinates near the subset $I^k \subset I^m$ which fixes $I^k$ pointwise and which sends $H_x $ to $\widetilde{H}_x$ for all $x \in Op(I^k)\setminus Op(\partial I^m)$. It follows from this observation that in the statement of Lemma \ref{calculation for almost tangent} we may replace the angle $\measuredangle(\tau,H)$ with the angle $\measuredangle(\tau  , I^k )$ formed by the distributions $\tau|_{I^k}$ and $TI^k$ along $I^k$. Indeed, $\measuredangle(\tau, I^k)=\inf \measuredangle(\tau|_{I^k}, \widetilde{H}|_{I^k})$, where the infimum is taken over all hyperplane fields $\widetilde{H} \subset TI^m$ such that $\bR^k \times 0^{m-k} \subset \widetilde{H}_x$ at all points $x \in I^m$. In the parametric case, we also allow the hyperplane field $\widetilde{H}_z \subset TI^m$ to vary with the parameter $z \in I^q$. Therefore in the statement of Lemma \ref{parametric calculation for almost tangent} we may replace the angle $\measuredangle(\tau_z,H)$ with the angle $\measuredangle(\tau_z , I^k)$.

\subsection{Reduction to the transverse case}\label{reduction to transverse case} We are ready to reduce Theorem \ref{localized holonomic approximation lemma for perp-holonomic sections} to the case where the hyperplane field is transverse to the subset $I^k \subset I^m$. Fix an arbitrary hyperplane field $\tau \subset TI^m$. Let $\sigma:I^m \to J^r(\bR^m,\bR^n)$ be a section such that $\sigma^{\perp}=0$ with respect to $\tau$ and such that $\sigma=0$ on $Op(\partial I^m)$. Fix $\varepsilon>0$ small, the desired accuracy for the $C^0$-approximation we must construct. Consider the cubical stratification of the subset $I^k \subset I^m$ in which the $j$-dimensional stratum consists of the union of the $j$-dimensional faces of the cubes
\[ Q_N(j_1, \ldots , j_k)=\Big[\frac{j_1}{N},\frac{j_1+1}{N}\Big] \times \cdots \times\Big[\frac{j_k}{N},\frac{j_k+1}{N}\Big] \subset I^k , \quad -N \leq j_1, \ldots, j_k < N.\]

Let $\lambda>0$ be small enough so that $C \, || \sigma ||_{C^r} \, \lambda < \varepsilon$, where $C>0$ is the constant which appears in the statement of Lemma \ref{calculation for almost tangent}. Choose $N \in \bN$ big enough so that for each cube $Q=Q_N(j_1, \ldots , j_k)$ we have $\measuredangle(\tau_x,\tau_y)<\lambda/2$ for all $x,y \in Op(Q)  \subset I^m$. Consider the polyhedron $R=\bigcup_{j=0}^kR^j \subset I^k$, where the stratum $R^j$ consists of the union of all the $j$-dimensional faces $F$ of the cubes $Q=Q_N(j_1, \ldots , j_k)$ such that $\measuredangle(\tau_x,F) \leq \lambda$ at all points $x \in F$. The hyperplane field $\tau$ is almost tangent to the faces $F$ in $R$ and transverse to all other faces $F$ in the cubical stratification of $I^k$ under consideration. Indeed, for all  faces $F$ not in $R$ we have $\measuredangle(\tau_x,F) \geq \lambda/2$ at all points $x \in F$.

We proceed inductively to construct a holonomic approximation of $\sigma$ over the cubical skeleton of $R$, working one face at a time as in Section \ref{Localization of the problem}. At each stage we  apply Lemma \ref{calculation for almost tangent}. Since $\measuredangle(\tau,F) \leq \lambda$ for each face $F$ under consideration, the resulting holonomic approximation has $C^0$-error $<\varepsilon$. On the remaining faces $F$ which are not in $R$, we have $\tau \pitchfork F$. The following result will be proved in Section \ref{holonomic approximation with controlled cutoff}, where we think of the hyperplane $V=0 \times \bR^{m-1} \subset \bR^m$ as a constant hyperplane field on $\bR^m$.
\begin{theorem}\label{final reformulation} Let $\sigma:I^m \to J^r(\bR^m, \bR^n)$ be a section such that the following properties hold.
\begin{itemize}
\item $\sigma=0$ on $Op(\partial I^m)$.
\item $\sigma^{\perp}=0$ on all of $I^m$ with respect to $V$.
\end{itemize}
Then there exists an isotopy $F_t:I^m \to I^m$ and a holonomic section $\hat{\sigma} :I^m \to J^r(\bR^m, \bR^n)$ such that the following properties hold.
\begin{itemize}
\item $\hat{\sigma}$ is $C^0$-close to $\sigma$ on $Op\big( F_1(I^k) \big)$.
\item $\hat{\sigma}^{\perp}$ is $C^0$-small on all of $I^m$.
\item $F_t$ is $C^0$-small.
\item $F_t=id_{I^m}$ and $\hat{\sigma}=0$ on $Op(\partial I^m)$.
\item $V$ is invariant under $F_t$. 
\end{itemize}
\end{theorem}
\begin{remark}\label{technical} The more accurate the $C^0$-approximation desired, the bigger the derivative $dF_t$ will need to be. However, for a fixed $C^0$-accuracy we can arrange it so that $\text{dist}_{C^0}(F_t,id_{I^m})$ is arbitrarily small and so that $F_t=id_{I^m}$ outside of an arbitrarily small neighborhood of $I^k$ in $I^m$ while keeping $||dF_t||_{C^0}$ uniformly bounded. This scale invariance follows from the explicit construction of $\hat{\sigma}$ and $F_t$ which is carried out in Section \ref{holonomic approximation with controlled cutoff}.
\end{remark}
Assuming Theorem \ref{final reformulation},  we continue the inductive process over the rest of the skeleton of $I^m$ to obtain a global holonomic $\varepsilon$-approximation of $\sigma$. Indeed, given a face $F$ not in $R$, we can approximate the hyperplane field $\tau|_F$ by a constant hyperplane field along $F$, see Figure \ref{piecewise constant}. We pay a price, of course, but the error can be made arbitrarily small by taking $N$ sufficiently big. We can therefore reduce the problem at each face $F$ to the local model considered in Theorem \ref{final reformulation}. Observe that the last property stated in Theorem \ref{final reformulation} and the a priori bound on $||dF_t||_{C^0}$ provided by Remark \ref{technical} are needed to show that after each step of the inductive process the approximation of $\tau$ by a piecewise-constant hyperplane field has not been ruined by the corresponding isotopy. To be more precise, the distorsion produced by each isotopy can be made arbitrarily small by taking $N$ sufficiently big. 

\begin{figure}[h]
\includegraphics[scale=0.62]{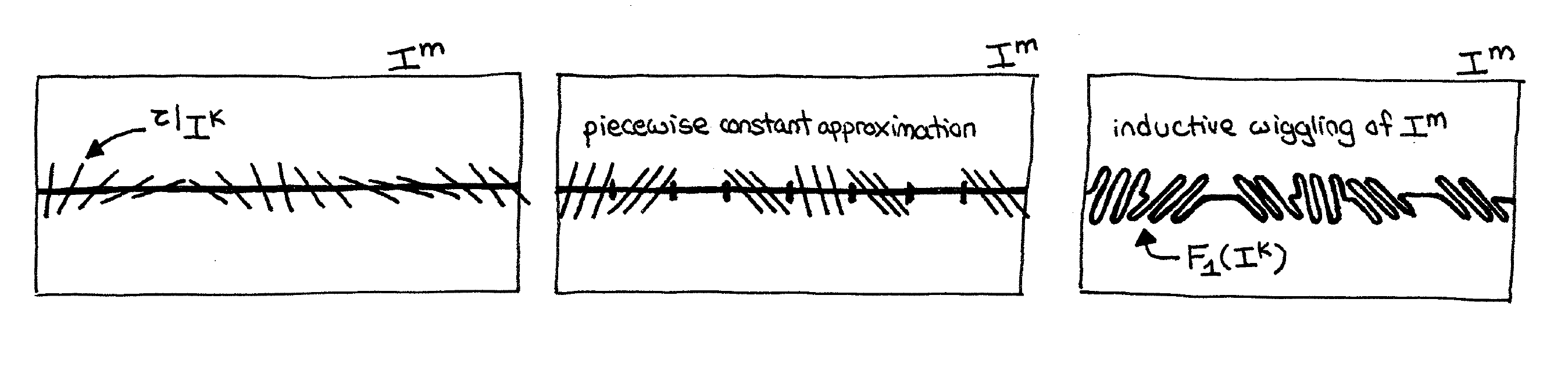}
\caption{Approximation of $\tau$ by a piecewise-constant hyperplane field and the corresponding wiggling.}
\label{piecewise constant}
\end{figure}

We have successfully reduced Theorem \ref{localized holonomic approximation lemma for perp-holonomic sections} to the transverse local model Theorem \ref{final reformulation} above. The same argument also works in families, using Lemma \ref{parametric calculation for almost tangent} instead of Lemma \ref{calculation for almost tangent}. Thus we can also reduce the parametric Theorem \ref{parametric localized holonomic approximation lemma for perp-holonomic sections} to a parametric transverse local model. The only difference in the reduction is that we must also subdivide the parameter space $I^q$, as well as the domain $I^m$, into small enough subcubes. The parametric version of Theorem \ref{final reformulation} reads as follows. 

\begin{theorem}\label{parametric final reformulation}
Let $\sigma_z:I^m \to J^r(\bR^m, \bR^n)$ be a family of sections parametrized by the unit cube $I^q$ such that the following properties hold.
\begin{itemize}
\item $\sigma_z=0$ on $Op(\partial I^m)$.
\item $\sigma^{\perp}_z=0$ on all of $I^m$ with respect to $V$.
\item $\sigma_z=0$ on all of $I^m$ for $z \in Op(\partial I^q)$.
\end{itemize}
Then there exists a family of isotopies $F^z_t:I^m \to I^m$ and a family of holonomic sections $\hat{\sigma}_z :I^m \to J^r(\bR^m, \bR^n)$ such that the following properties hold.
\begin{itemize}
\item $\hat{\sigma}_z$ is $C^0$-close to $\sigma_z$ on $Op\big( F^z_1(I^k) \big)$.
\item $\hat{\sigma}^{\perp}_z$ is $C^0$-small on all of $I^m$.
\item $F^z_t$ is $C^0$-small.
\item $F^z_t=id_{I^m}$ and $\hat{\sigma}_z=0$ on $Op(\partial I^m)$.
\item $F^z_t=id_{I^m}$ and $\hat{\sigma}_z=0$ on all of $I^m$ for $z \in Op(\partial I^q)$.
\item $V$ is invariant under $F^z_t$.
\end{itemize}
\end{theorem}

We note that there also is an priori bound on $||dF^z_t||_{C^0}$ depending on the desired accuracy of the $C^0$-approximation, just as in Remark \ref{technical}. We have now completed all preparatory reductions.

\section{Holonomic approximation with controlled cutoff}\label{holonomic approximation with controlled cutoff}
\subsection{The transverse local model}\label{transverse local model} We begin by establishing some simple estimates which will be crucial in the quantitative holonomic approximation process described below. We exploit the concreteness of the local models considered in Theorems \ref{final reformulation} and \ref{parametric final reformulation} by writing down the main objects explicitly, differentiating them by hand and thereby deducing the necessary bounds. We once again use $C>0$ to denote a constant, which only depends on $m$ and $r$, but which will be replaced with a bigger such constant whenever necessary.

Consider a section $\sigma:I^m \to J^r(\bR^m, \bR^n)$ such that $\sigma^{\perp}=0$ with respect to the constant hyperplane field $V=0 \times \bR^{m-1}$. In the spirit of Section \ref{Taylor polynomials}, we can give an explicit description of $\sigma$. Each $r$-jet $\sigma(x) \in J^r(\bR^m, \bR^n)$ at a point $x =(x_1, \ldots, x_m) \in I^m$ corresponds to a germ 
\[ y \mapsto h_x(y)=(y_1-x_1)^r \cdot \big( v_1(x), \ldots , v_n(x) \big) \in \bR^n, \qquad y = (y_1, \ldots, y_m)\in Op(x) \subset I^m.\]

\begin{figure}[h]
\includegraphics[scale=0.6]{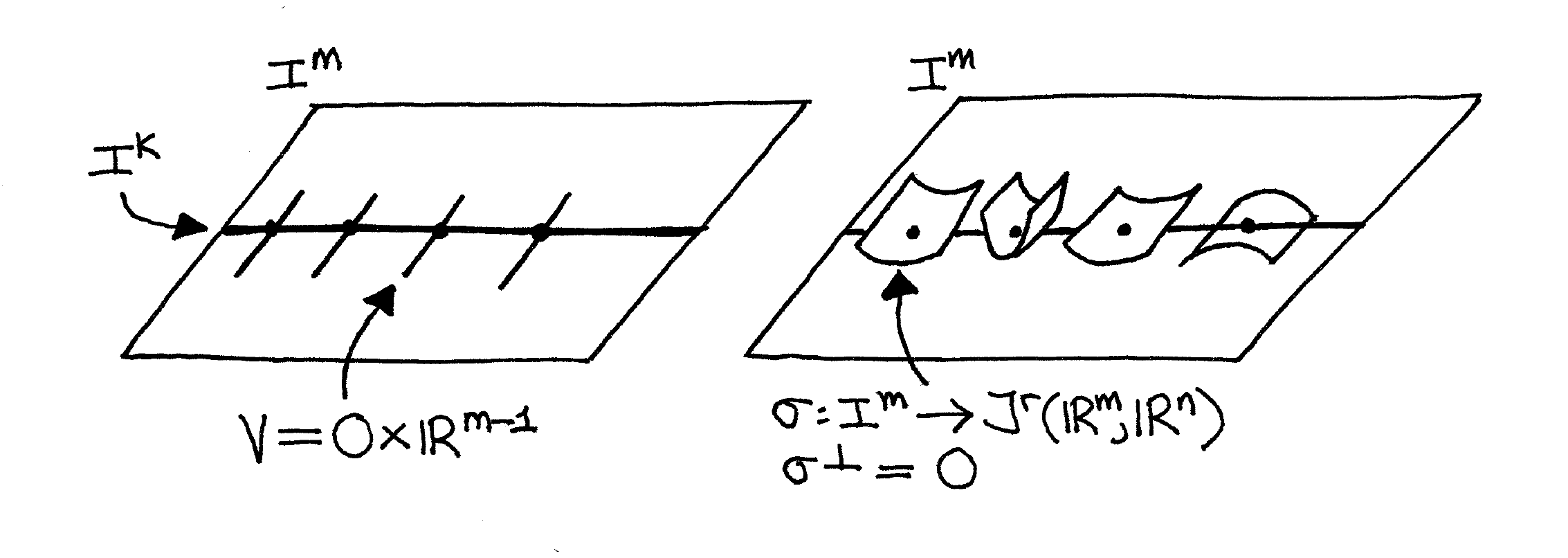}
\caption{A typical section $\sigma$ which is primitive with respect to the hyperplane field $V$.}
\label{V-primitive}
\end{figure}

 If $\sigma=0$ on $Op(\partial I^m)$, then the function $v=(v_1, \ldots, v_n):I^m \to \bR^n$ also satisfies $v=0$ on $Op(\partial I^m)$. We must control the derivatives of the function $h(x,y)=h_x(y)$ with respect to both $x$ and $y$. For this purpose, let $\alpha$ and $\beta$ be multi-indices such that $|\alpha|+ |\beta| \leq r$. Write $\alpha=(\xi,\gamma)$ where $\xi$ consists of $M$ indices $\xi_j=1$ and $\gamma$ consists only of indices $1 <\gamma_j  \leq m$. Similarly, write $\beta=(\zeta, \mu)$, where $\zeta$ consists of $N$ indices $\zeta_j=1$ and $\mu$ consists only of indices $1 <\mu_j \leq m$. We compute the following formula for the derivatives of $h(x,y). $
\[ \frac{\partial}{\partial x_\alpha} \frac{\partial}{\partial y_\beta} h|_{(x,y)} = \sum_{j=1}^M  \sum_{i=1}^{N}{ M \choose j} { N \choose i} (-1)^j \frac{r!}{(r-j-i)!}(y_1-x_1)^{r-j-i} \cdot   \big(\partial^{M-j-i}_1 \partial_\mu \partial_\gamma \, v \big)  (x).\]

We therefore obtain the estimate
\[ || \frac{\partial}{\partial x_\alpha} \frac{\partial}{\partial y_\beta} h || \leq  C \, || \sigma ||_{C^r} \, \delta^{r-(M+N)} \quad \text{for} \quad y \in \{ |y_1-x_1| < \delta \} \subset I^m. \]

\begin{figure}[h]
\includegraphics[scale=0.5]{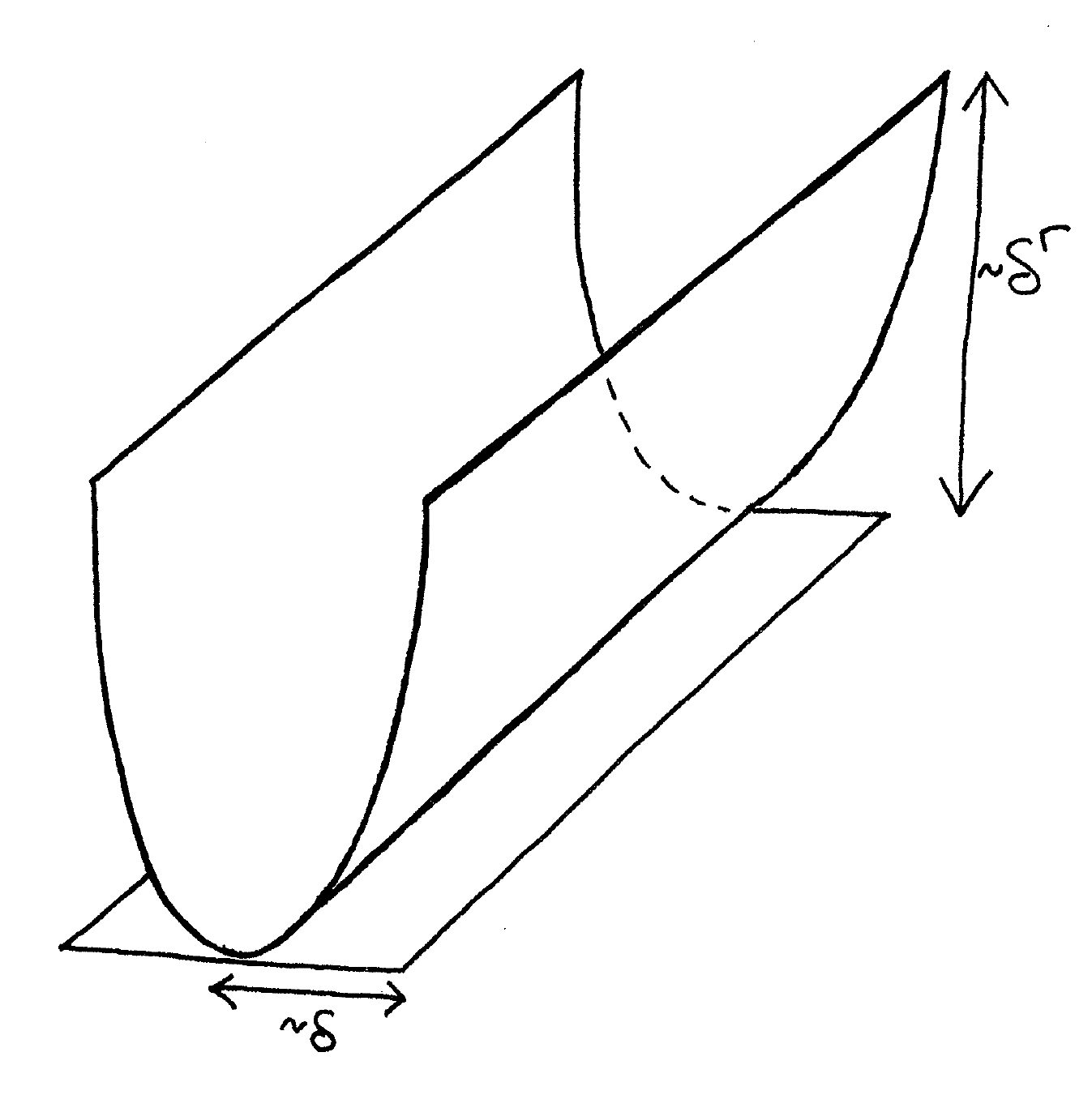}
\caption{Example: the estimate $|| h || \leq C || \sigma ||_{C^3} \delta^r$ .}
\label{estimate}
\end{figure}

We next give a local model for the wiggling. Many other choices work just as well, of course, but we want to write down an explicit model for concreteness. Given $\varepsilon, \delta>0$ small such that $\delta<< \varepsilon$ to an extent which will be made precise later, consider the sinusoidal curve
\[ w(u)= \frac{\varepsilon}{2} \,\sin\big(  \frac{ \pi u }{ 2 \delta} \big) , \qquad u \in \bR . \]

Fix a cutoff function $\psi:\bR \to \bR$ such that $\psi(u)=0$ for $|u|<1/2$ and $\psi(u)=1$ for $|u|>3/4$. For $\varepsilon>0$ small enough so that $\text{supp}(\sigma) \subset [-1+\varepsilon, 1-\varepsilon]^m$, define an isotopy $F_t:I^m \to I^m$ by the formula
\[ F_t(x_1, \ldots, x_m)=\big(x_1, \ldots, x_m+\varphi_t(x) \big), \qquad \varphi_t(x)=t \, \psi\left(\frac{1-|x_1|}{\varepsilon}\right) \cdots \,  \psi \left(\frac{1-|x_m|}{\varepsilon}\right) \, w(x_1). \]

\begin{figure}[h]
\includegraphics[scale=0.6]{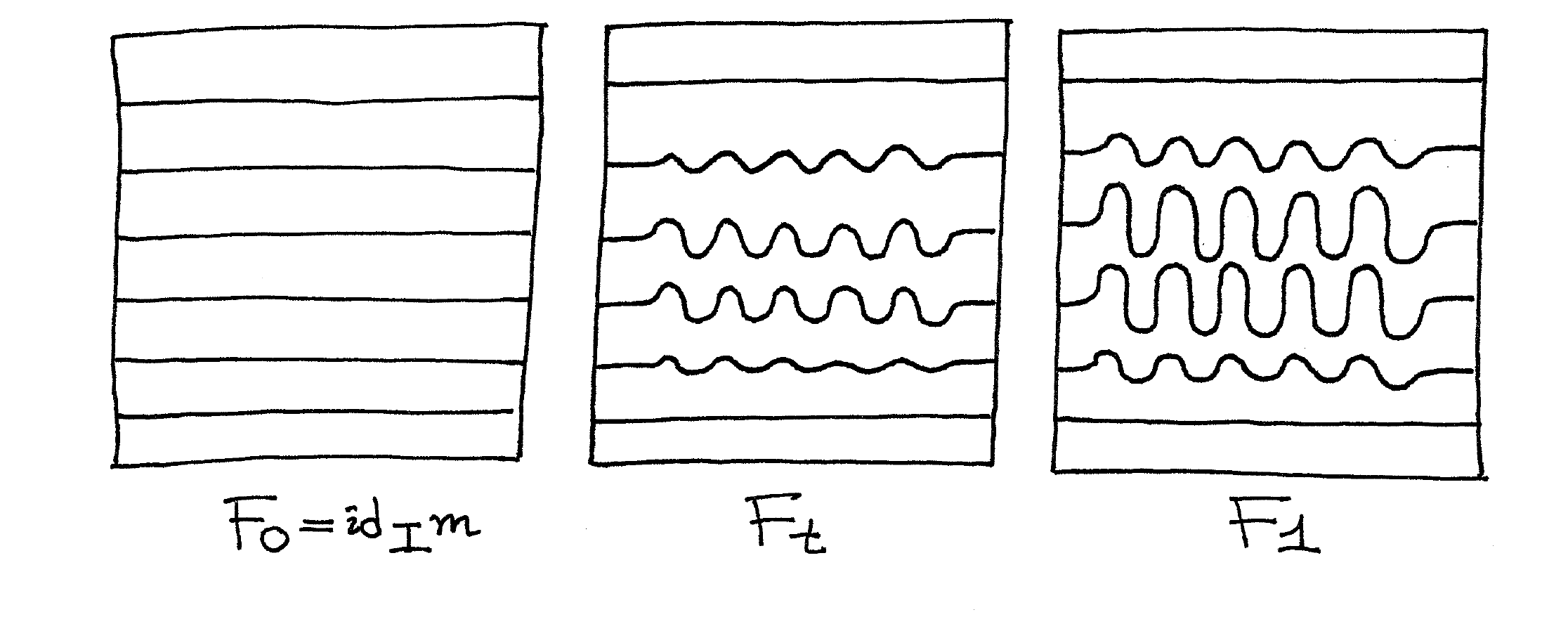}
\caption{The isotopy $F_t$.}
\label{holonomic approximation}
\end{figure}

Observe that $\text{dist}_{C^0}(F_t,id_{I^m})< \varepsilon$. We also have the following estimate for the derivative $dF_t$ of the isotopy $F_t$, where we note that the ratio $\varepsilon/\delta$ will typically be very big but remains invariant by a simultaneous scaling of $\varepsilon$ and $\delta$.
\[ ||dF_t||_{C^0} \leq C \, \, \frac{ \varepsilon }{ \delta } .\]

\begin{figure}[h]
\includegraphics[scale=0.6]{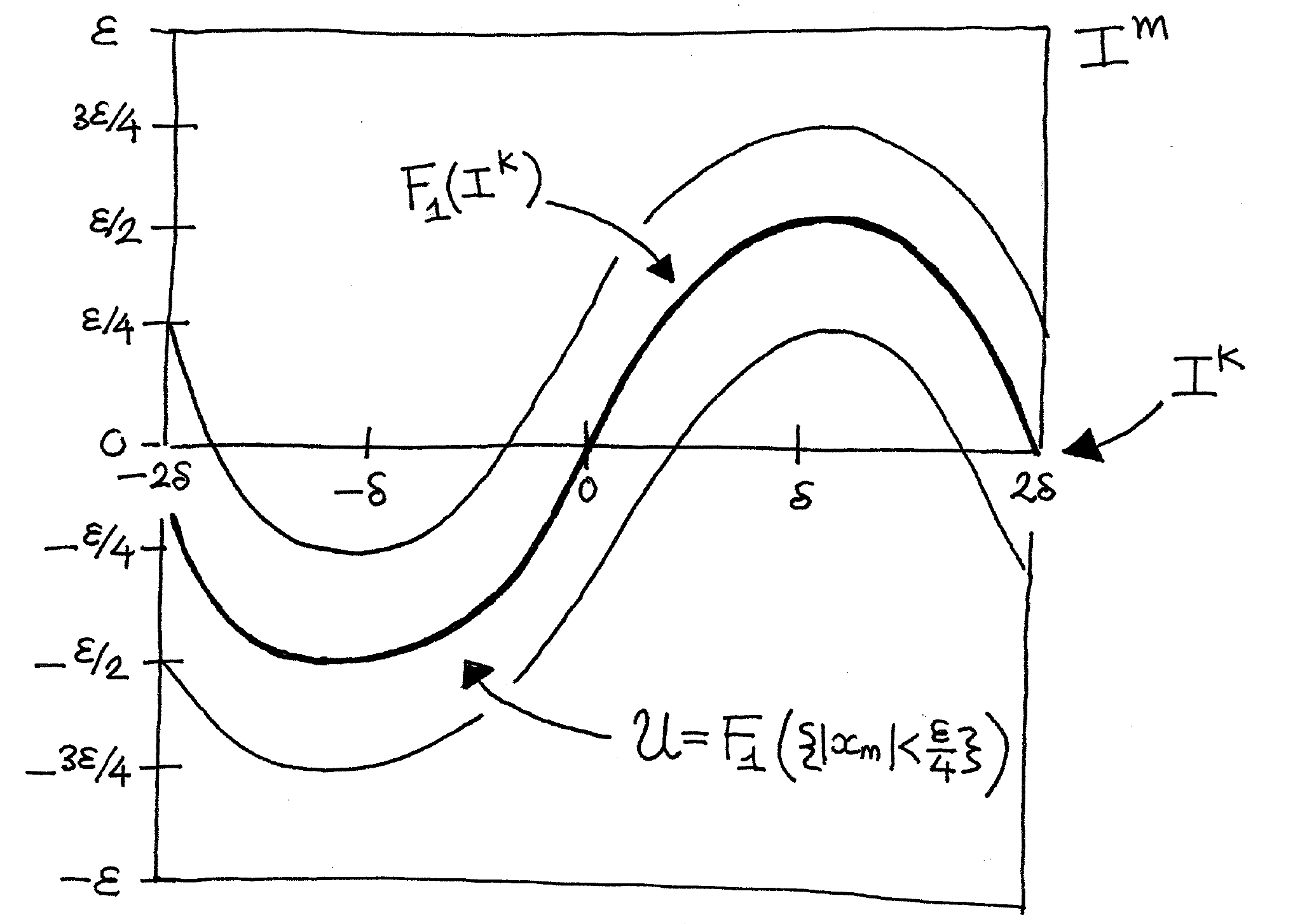}
\caption{The sinusoidal subset $U \subset I^m$ in a neighborhood of the origin.}
\label{sinusoidal}
\end{figure}

In Section \ref{holonomic approximation process} below we apply the same method of proof as Eliashberg and Mishachev in \cite{EM02}  to produce a holonomic approximation $j^r(g)$ of the section $\sigma$. The domain of definition of the function $g$ is the wiggled neighborhood of $I^k \subset I^m$ given by $U=F_t\big( \{ |x_m|<\varepsilon/4 \}\big) \subset I^m$. To extend our holonomic approximation to the whole of $I^m$, we multiply the function $g:U \to \bR^n$ by a cutoff function supported in $U$. We must control the derivatives of such a cutoff function, so we now write down an explicit model together with the appropriate estimate. 

 In terms of the function $\psi$ fixed above, let $\phi: I^m \to \bR$ be given by $\phi(x)=1-\psi\big(4|y_m|/\varepsilon\big)$, where $F_1(y)=x$. Note that $\phi=1$ near $F_1(I^k)$ and that $\text{supp}(\phi) \subset U$, see Figure \ref{cutoff}. The following bound holds for the derivatives of $\phi$. Let $\alpha$ be a multi-index of order $|\alpha| \leq r$. Write $\alpha=(\beta, \gamma)$, where $\beta$ consists of $N$ indices $1 < \beta_j \leq m$ and $\gamma$ consists of $M$ indices $\gamma_j=1$. Then we have
\[ | \partial_\alpha \phi | \leq C \, \frac{1}{ \varepsilon^N \delta^M} .  \]

\begin{figure}[h]
\includegraphics[scale=0.6]{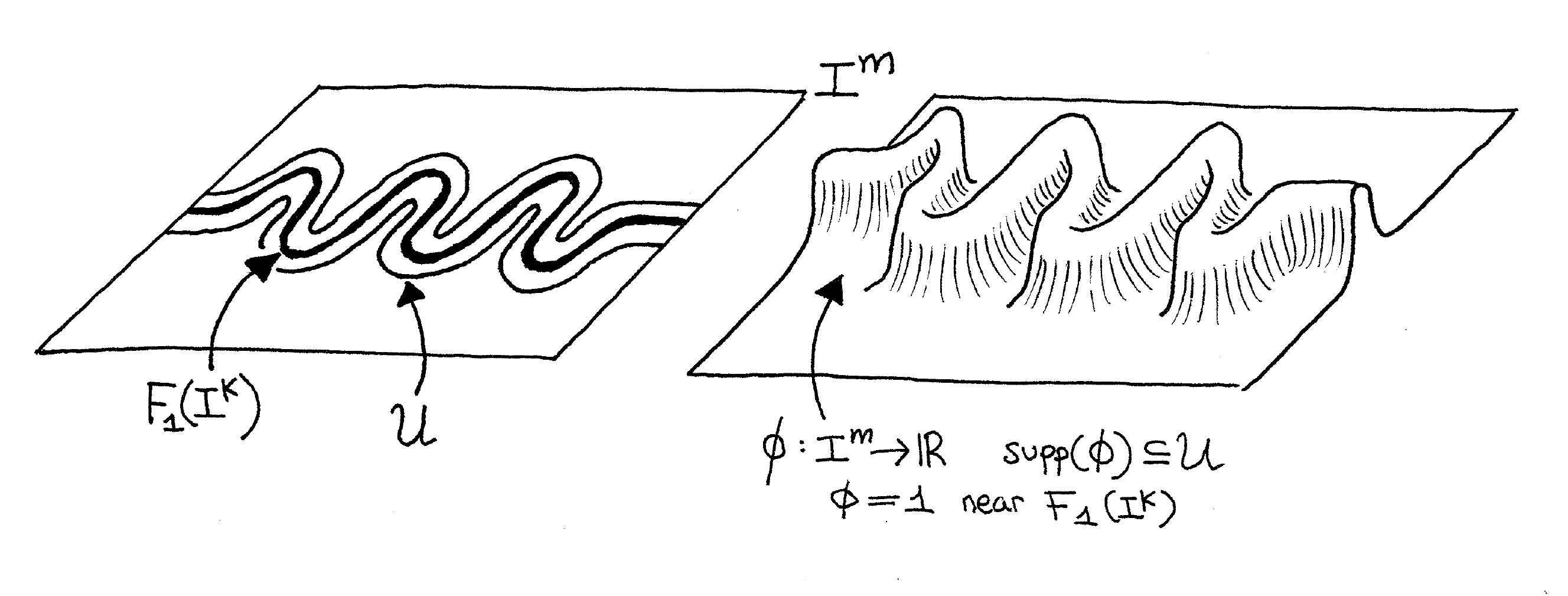}
\caption{The cutoff function $\phi$.}
\label{cutoff}
\end{figure}

\subsection{The holonomic approximation process}\label{holonomic approximation process}
We are ready to construct the holonomic approximation $\hat{\sigma}$ of $\sigma$. We use the isotopy $F_t:I^m \to I^m$ and the cutoff function $\phi:I^m \to \bR$ defined in Section \ref{transverse local model}, which depend on two parameters $\varepsilon$ and $\delta$. We will obtain an arbitrarily good $C^0$-approximation $\hat{\sigma}$ by choosing $\varepsilon,\delta>0$ arbitrarily small such that the ratio $\delta / \varepsilon$ is also arbitrarily small. Fix a function $\eta:\bR \to \bR$ such that 
\begin{itemize}
\item $\eta(u)=-1$ for $u<-1$, 
\item $1 \leq \eta(u) \leq 1$ for $-1 \leq u \leq 1$, 
\item $\eta(u)=1$ for $u >1$.
\end{itemize}
 We construct $\hat{\sigma}$ by writing down an explicit formula on each of the rectangles 
\[ R_j= \big[(2j-1)\delta,(2j+1)\delta\big] \times I^{m-2}\times[-\varepsilon,\varepsilon ] \subset I^m\]
such that $R_j$ is contained in the support of $\sigma$. Suppose first that $j \in \bN$ is even. Define a function $g:R_j \to \bR^n$ by
\[ g(x)=h\big(p(x),x\big), \qquad \text{where} \quad p(x)=\big((2j\delta)+\delta \eta(4x_m/\varepsilon),x_2, \ldots, x_{m-1},0 \big), \qquad x=(x_1, \ldots , x_m) \in R_j.\]

Let $b(u)=(2j\delta)+ \delta \eta(4u/ \varepsilon)$, so that $p(x)=\big(b(x_m),x_2, \ldots, x_{m-1},0\big)$. We note for future reference the following bound on the derivatives of the function $b$.
\[ | b^{(i)}| \leq C \, \frac{\delta }{\varepsilon^i}. \]

\begin{figure}[h]
\includegraphics[scale=0.6]{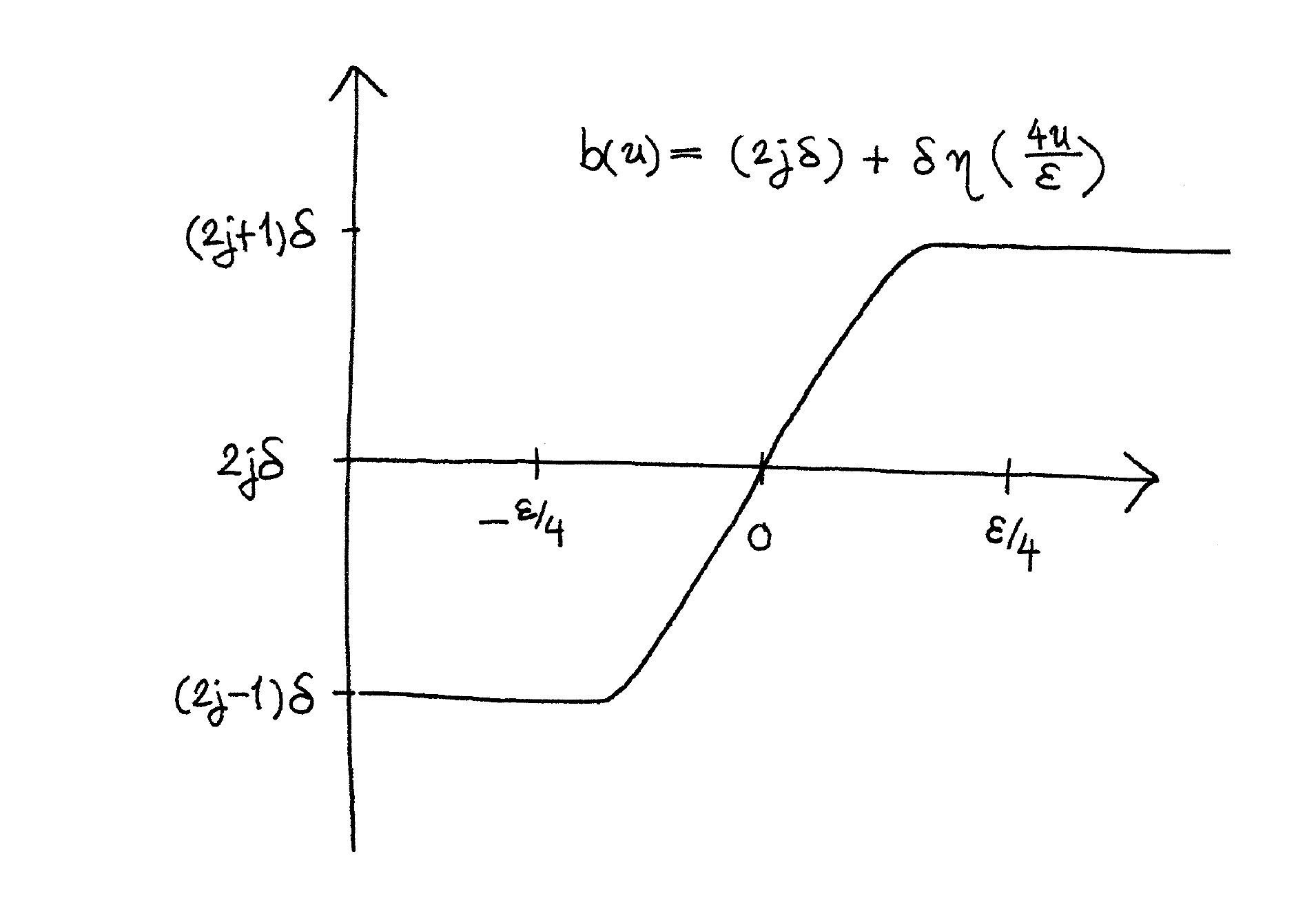}
\caption{The basepoint-interpolating function $b$.}
\label{gluing}
\end{figure}

\begin{remark}Observe that:
\begin{enumerate}
\item On $R_j \cap \{ x_m<-\varepsilon/4 \}$ we have $p(x)=\big( (2j-1)\delta, x_2, \ldots, x_{m-1},0 \big)$
\item On $R_j \cap \{ x_m> \varepsilon/4\}$ we have $p(x)=\big( (2j+1)\delta, x_2, \ldots, x_{m-1},0 \big)$.
\item On $R_j \cap \{|x_m|<\varepsilon/4\}$ the point $p(x)$ interpolates between $\big( (2j-1)\delta, x_2, \ldots, x_{m-1},0 \big)$ and $\big( (2j+1)\delta, x_2, \ldots, x_{m-1},0 \big)$.
\end{enumerate}

Similarly, if $j \in \bN$ is odd, we define a function $g:R_j \to \bR^n$ by
\[ g(x)=h\big(p(x),x\big), \qquad \text{where} \quad p(x)=\big((2j\delta)-\delta \eta(4x_m/\varepsilon),x_2, \ldots, x_{m-1},0 \big), \qquad x=(x_1, \ldots , x_m) \in R_j.\]

The difference in the sign corresponds to the fact that on the interval $\big[(2j-1)\pi/2, (2j+1)\pi/2\big]$ the function $u \mapsto \sin(u)$ is increasing for $j$ even and decreasing for $j$ odd. Note that the locally defined functions $g:R_j \to \bR^n$ do not glue together on $\{|x_m|<\varepsilon\}=\bigcup_j R_j \subset I^m$. However, they do glue together on the wiggled neighborhood $U=F_1\big( \{ |x_m|<\varepsilon/4 \} \big)$ of $F_1(I^k)$, see Figure \ref{gluing}. We therefore obtain a globally defined function $g:U \to \bR^n$.  

\begin{figure}[h]
\includegraphics[scale=0.6]{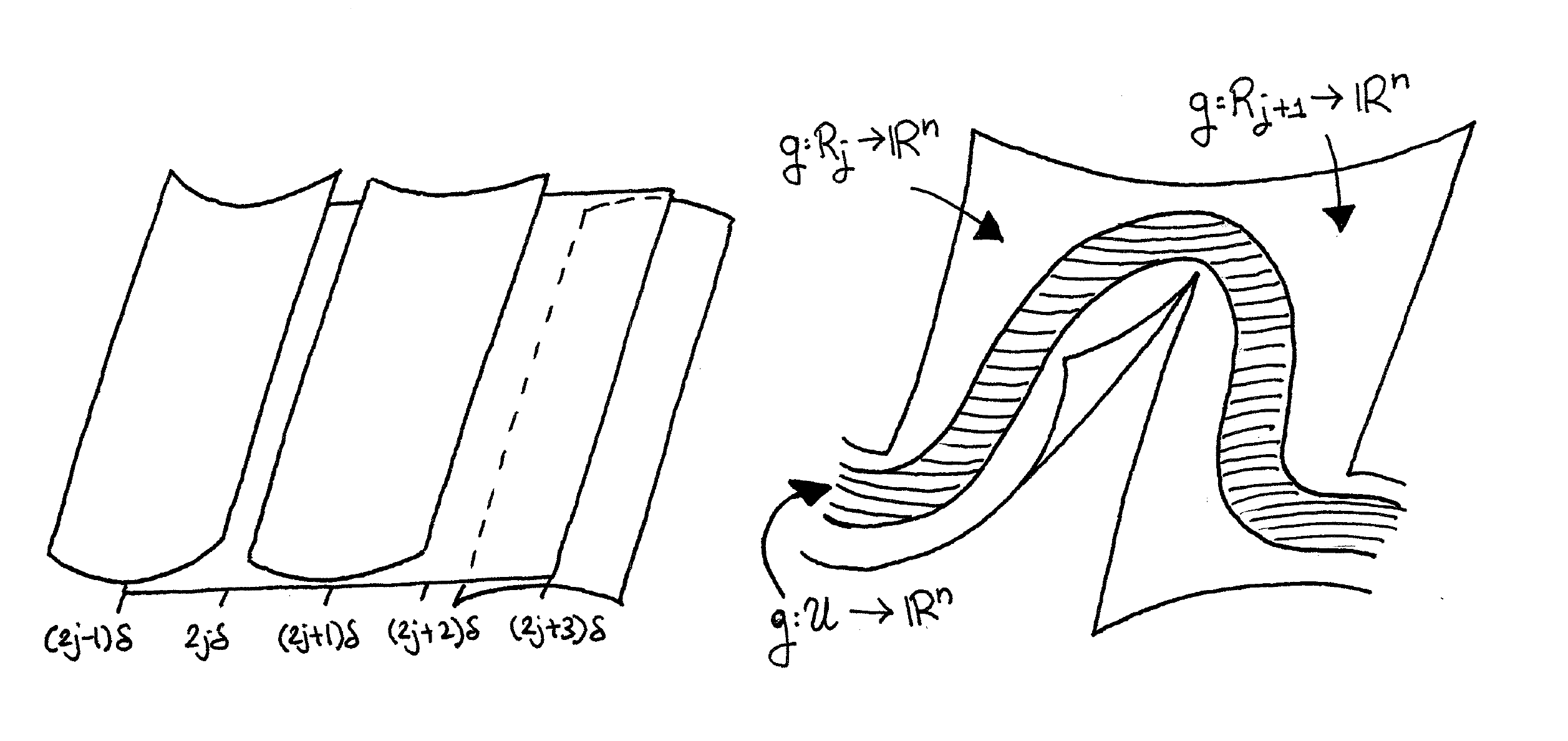}
\caption{The glued up function $g:U \to \bR^n$.}
\label{gluing}
\end{figure}

Set $f(x) = \phi(x) \cdot g(x)$. Since $\text{supp}(\phi) \subset U$, the function $f:I^m \to \bR$ is defined on all of $I^m$. The holonomic section $\hat{\sigma}=j^r(f)$ is the desired holonomic approximation to $\sigma$. It remains to prove that for an adequate choice of the parameters $\varepsilon$ and $\delta$ all of the properties listed in Theorem \ref{final reformulation} are satisfied. The last three properties can be verified by inspection. In the next section we carry out the calculation required to establish the other two.
\end{remark}

\subsection{Conclusion of the proof}\label{conclusion of the proof}  Let $\alpha$ be a multi-index of order $|\alpha| \leq r$. Write $\alpha=(\beta, \gamma , \xi)$, where $\beta$ consists of $I$ indices $\beta_j=1$, $\gamma$ consists of $J$ indices $1 < \gamma_j <m$ and $\xi$ consists of $K$ indices $\xi_j=m$. Then we compute 
\[ \frac{\partial }{\partial x_\alpha}g|_x=\sum \big( \frac{\partial^i}{\partial x_1^i} \frac{\partial}{\partial x_\gamma} \frac{\partial^I}{\partial y_1^I} \frac{\partial^j}{\partial y_m^j}  h|_{( p(x),x ) }   \big) b^{(k_1)}(x_m) \cdots b^{(k_i)}(x_m),\]
where the sum is taken over all non-negative integers such that $j+k_1 + \cdots + k_i=K$. From the estimates established in Section \ref{transverse local model} we deduce that
\[ || \frac{\partial}{\partial x_\alpha} g || \leq C \, || \sigma ||_{C^r}  \sum  \delta^{r-(i+I)} \big(\frac{ \delta }{ \varepsilon^{k_1}} \big) \cdots\big(\frac{ \delta }{ \varepsilon^{k_i}} \big)= C \,  || \sigma ||_{C^r} \sum \, \frac{\delta^{r-I}}{\varepsilon^{k_1+ \cdots + k_i}} \leq C \, || \sigma ||_{C^r} \frac{\delta^{r-I}}{\varepsilon^K } .\] 

Suppose first that $\alpha$ is a multi-index such that $I<r$. Observe that $\delta^{r-I} \leq \delta^{K +J} \leq \delta^K$. It follows that we have
\[ || \frac{\partial}{\partial x_{\alpha}}g || \leq   C\, || \sigma ||_{C^r} \frac{ \delta }{ \varepsilon} \]
and therefore we can make this derivative arbitrarily small by ensuring that the ratio $\delta / \varepsilon$ is arbitrarily small. When $I=r$, the multi-index $\alpha=(1, \ldots , 1)$ corresponds to the pure $r$-th derivative $\partial^r_1$. Observe that in this case the sum collapses to 
\[ \frac{\partial^r}{\partial x_1^r }g |_x = \frac{\partial^r}{\partial y_1^r}h |_{( p(x), x )}. \]
Since $||x-p(x)|| \leq C \, \varepsilon$, it follows that
\[ || \frac{\partial^r}{\partial y_1^r}h|_{(p(x),x)} - \frac{\partial^r}{\partial y_1^r }h|_{(x,x)} || \leq C \, || \sigma ||_{C^r} \varepsilon. \]
Hence we deduce the inequality
\[ \text{dist}_{C^0}(\hat{\sigma}, \sigma ) \leq C \, || \sigma ||_{C^r} \, \big(\varepsilon + \frac{ \delta }{ \varepsilon } \big)   \]
on the smaller neighborhood of $F_1(I^k)$ in $U$ where $\phi=1$, so that $f=g$ and $\hat{\sigma}=j^r(g)$. This proves that the $C^0$-approximation can be made arbitrarily accurate by choosing $\varepsilon, \delta>0$ arbitrarily small such that the ratio $\delta / \varepsilon$ is also arbitrarily small. 

It remains to show that for such a choice of $\varepsilon$ and $\delta$ we can also ensure that $\hat{\sigma}^{\perp}$ is small on all of $I^m$. We again must explicitly compute some derivatives. Let $\alpha$ be a multi-index of order $|\alpha| \leq r$. Then we have 
\[ \frac{\partial}{\partial x_\alpha} f = \sum_{\beta + \gamma = \alpha}( \frac{\partial}{\partial x_\beta} \phi )( \frac{\partial}{\partial x_\gamma} g ). \]

Suppose that $\alpha_j=1$ for exactly $N$ indices, where $N \leq r$. Invoking the estimates established in Section \ref{transverse local model} we deduce that 
\[ || \frac{\partial }{\partial x_\alpha} f ||  \leq C\, || \sigma ||_{C^r}\Big( \,  \frac{\delta}{\varepsilon}\,  \Big)^{r-N} . \]
It follows that with the exception of the case $N=r$ we have 
\[ || \frac{\partial }{\partial x_\alpha} f ||  \leq C\,||\sigma||_{C^r} \, \frac{\delta}{\varepsilon} . \]
Since the $\perp$-jet $\hat{\sigma}^{\perp}$ consists of all derivatives $\partial_\alpha f$ for multi-indices $\alpha$ of order $|\alpha| \leq r$ such that $N<r$, we obtain the inequality 
\[ ||\hat{ \sigma}^{\perp} ||_{C^0} \leq C \, || \sigma ||_{C^r} \, \frac{ \delta }{ \varepsilon }. \]

This concludes the proof of Theorem \ref{final reformulation}.

\subsection{The parametric case} The above calculation also works in families, by adding a parameter everywhere. We spell out the details for completeness. Let $\sigma_z:I^m \to J^r(\bR^m, \bR^n)$ be a family of sections parametrized by the unit cube $I^q$ such that $\sigma_z^{\perp}=0$ with respect to $V=0 \times \bR^{m-1}$, such that $\sigma_z=0$ on $Op(\partial I^m)$ and such that $\sigma_z=0$ on all of $I^m$ for $z \in Op(\partial I^q)$. We  think of $\sigma_z$ as a family of germs $(z,x,y) \mapsto h(z,x,y) \in \bR^m $, where $x \in I^m$, $y \in Op(x)$ and $z \in I^q$.  

We can define a family of functions $g_z$ as before by setting $g_z(x)=h\big(z,p(x),z\big)$ on each rectangle $R_j$. The domain of definition of $g_z$ is $U_z=F^z_1(\{|x_m|<\varepsilon/4\}) \subset I^m$ , where for each $z=(z_1, \ldots , z_q) \in I^q$ we have an isotopy $F^z_t:I^m \to I^m$ given by
\[ F^z_t(x_1, \ldots, x_m)=\big(x_1, \ldots, x_m+\varphi^z_t(x) \big), \] \[ \varphi^z_t(x)=t \, \psi\left(\frac{1-|x_1|}{\varepsilon}\right) \cdots \,  \psi \left(\frac{1-|x_m|}{\varepsilon}\right) \,  \psi \left(\frac{1-|z_1|}{\varepsilon}\right) \cdots  \psi \left(\frac{1-|z_q|}{\varepsilon}\right) \, w(x_1).  \] 

We use the same functions $\psi$ and $w$ as in the non-parametric case. We also use the corresponding family of cutoff functions $\phi_z:I^m \to \bR$, satisfying $\phi_z=1$ near $F^z_1(I^k)$ and $\text{supp}(\phi_z) \subset  U_z$, which are given by
\[  \phi_z(x)=1-\psi(4|y_m|/\varepsilon), \qquad F^z_1(x)=y.\]

Set $f_z(x)=\phi_z(x) \cdot g_z(x)$ and $\hat{\sigma}_z=j^r(f_z)$ to obtain the desired holonomic approximation. The computation carried out in Section \ref{holonomic approximation process} shows that by taking $\varepsilon$ and $\delta $ arbitrarily small such that the ratio $\delta / \varepsilon$ is also arbitrarily small, all of the properties stated in Theorem \ref{parametric final reformulation} are satisfied.

\end{document}